\documentclass{amsart}
\usepackage{amsmath}
\usepackage{amsthm}
\usepackage{graphicx}
\usepackage{latexsym}
\usepackage{color}
\usepackage{amscd}
\usepackage[all]{xy}
\usepackage{enumerate}
\usepackage{hyperref}
\usepackage{subfig}
\usepackage{soul}
\usepackage{comment}
\usepackage{epsfig}
\usepackage{epstopdf}
\usepackage{pinlabel}
\usepackage{multirow}

\usepackage[pagewise]{lineno}
 
\newtheorem{thm}{Theorem}
\newtheorem*{thmm}{Main Theorem}

\newtheorem*{claim*}{Claim}
\newtheorem{theorem}[thm]{Theorem}
\newtheorem{lemma}[thm]{Lemma}
\newtheorem{corollary}[thm]{Corollary}
\newtheorem{proposition}[thm]{Proposition}

\theoremstyle{definition}

\newtheorem*{definition*}{Definition}

\newtheorem{remark}[thm]{Remark}

\newcommand{\al}{\alpha}
\newcommand{\be}{\beta}
\newcommand{\ga}{\gamma}
\newcommand{\kai}{\mathfrak{X}}

\newcommand{\cn}{\mathcal{C}(N_{g,n})}

\newcommand{\CC}{\mathcal{C}}

\newcommand{\gap}{\mathcal{A}(P)}
\newcommand{\gapf}{\mathcal{A}(\phi(P))}
\newcommand{\Mod}{{\rm Mod}}
\newcommand{\Aut}{{\rm Aut}}
\newcommand{\dela}{\delta_{\alpha}}
\newcommand{\delb}{\delta_{\beta}}
\newcommand{\kaign}{\kai_{g,n}}

\newcommand{\Ngn}{N_{g,n}}

\newcommand{\kgn}{\mathfrak{X}_{g,n}}

\newcommand{\link}{\mbox{Link}}
\begin{document}

\title[Finite Rigid Sets in Curve Complexes]
{Finite Rigid Sets in Curve Complexes of Non-Orientable Surfaces}
\author[S. Ilb{\i}ra ]{SABAHATT\.{I}N ILBIRA}
\address{Department of Mathematics, Amasya University, 05000 Amasya, Turkey}
\email{sabahattin.ilbira@amasya.edu.tr}

\author[M. Korkmaz]{Mustafa Korkmaz}
\address{Department of Mathematics, Middle East Technical University, 06800 Ankara, Turkey}
\email{korkmaz@metu.edu.tr}

\date{\today}

\begin{abstract}
A rigid set in a curve complex of a surface is a subcomplex 
such that every locally injective simplicial map from the set
into the curve complex is induced by a homeomorphism of the surface. 
In this paper, we find finite rigid sets 
in the curve complexes of connected non-orientable surfaces 
of genus $g$ with $n$ holes for $g+n \neq 4$. 
\end{abstract}

\maketitle

\setcounter{secnumdepth}{3}
\setcounter{section}{0}

\section{Introduction} 
 
 The curve complex on a surface, introduced by Harvey~\cite{harv81} in order to study the boundary structure 
 of the Teichm\"uller space, is an abstract simplicial complex whose vertices are isotopy classes of simple
 closed curves.  The mapping class group of the surface acts on it by simplicial 
 automorphisms in a natural way, inducing an isomorphism, except for a few sporadic cases, from the mapping class group
 onto the group of simplicial automorphisms of the curve complex. These results are proved by
 Ivanov~\cite{ivanov97}, Korkmaz~\cite{korkmaz99} and Luo~\cite{luo00} in the orientable case,
 and by Atalan-Korkmaz \cite{atako14} in the non-orientable case.
 
 In the case the surface is orientable, it was shown by Shackleton~\cite{shack07} that a locally injective 
 simplicial map of the curve complex is induced by a homeomorphism.  Aramayona-Leininger~\cite{aramlei13}  
 showed that the curve complex contains a finite subcomplex $\kai$ such that any 
 locally injective simplicial map from $\kai$ into the curve complex is induced by a 
 homeomorphism. They called these sets \textit{finite rigid sets}.

The purpose of this paper is to prove the non-orientable version of  Aramayona-Leininger result: We prove that
if  $g+n\neq 4$ then the curve complex of a connected non-orientable surface $N$ of genus $g$ with $n$ 
boundary components contains a finite rigid set. More precisely, we prove the following theorem:

 \begin{thmm} \label{mainthmintro} 
       Let $g+n\neq 4$ and let $N$ be a connected non-orientable  surface of genus $g$ with $n$ 
       boundary components. There is a finite subcomplex $\kai_{g,n}$ of the curve complex \
       $\CC(N)$ of $N$  such that for any locally injective 
       simplicial map $\phi:\kai_{g,n} \rightarrow \CC(N)$ there is a homeomorphism $F:N\rightarrow N$
       with $\phi(\gamma)=F(\gamma)$ for all $\gamma\in \kaign$. The homeomorphism is unique 
       up to an involution.
 \end{thmm}

This theorem is restated and proved as Theorem~\ref{mainxx} in Section~\ref{sec=mainthm}. The proof 
of it is by induction on the genus $g$. 
We also show at the end of the paper that in the case $g+n=4$ the subcomplex  $\kai_{g,n} $ 
we consider is not rigid.
 
 After Aramayona-Leininger introduced the finite rigid sets 
 in the curve complex, Maungchang~\cite{maung13} obtained 
 finite rigid subgraphs of the pants graphs of punctured spheres. 
 Aramayona-Leininger in~\cite{aramlei16}  
 constructed a sequence of finite rigid sets exhausting 
 curve complex of an orientable surface, and Hern\'andez-Hern\'andez~\cite{hernexhaustion16} 
 obtained an exhaustion of the curve complex via rigid expansion. 
 In particular, they show that for an orientable surface $S$,  there exists a sequence of finite rigid sets  
 each having trivial pointwise stabilizer in the mapping class group so that the union of these sets is the 
 curve complex of $S$.

 For non-orientable surfaces, the curve complex was first studied by Scharlemann \cite{schar82}. 
  Atalan-Korkmaz~\cite{atako14} proved that 
 the natural map from the mapping class group $\Mod(N)$ of a connected non-orientable surface 
 $N$ of genus $g$ with $n$ holes to the automorphism group $\Aut(\CC(N))$ 
 of the curve complex  of $N$ is an isomorphism for $g+n\geq5$,
 Irmak generalized this result by showing that superinjective simplicial maps and injective simplicial maps $\CC(N)\rightarrow\CC(N)$ are induced by homeomorphisms $N\rightarrow N$ (cf. ~\cite{irmak14edge} and \cite{irmak14sup}).  
 
 Here is an outline of the paper. Section~\ref{prelim} contains the relevant definitions and the background. 
 The model of the non-orientable surface and the curves used in the paper are introduced in this section. 
 Section~\ref{sec:subcomplex} introduces the finite subcomplex $\kaign$ which will be showed to be rigid.  
 Section~\ref{sec:types} focuses on the topological types of vertices in the finite subcomplex $\kaign$ 
 and shows that topological types of many elements of $\kaign$ are preserved under 
 locally injective simplicial maps $\kaign\to \CC(N)$.   
 Section~\ref{sec=mainthm} is devoted to the proof of the Main Theorem: The proof is by induction on the genus of the surface.  It is shown in the last section,
 Section~\ref{sec:son}, that the complex $\kaign$ is not rigid in the case $g+n=4$.
 
 The main result of this paper was part of the first author's Ph.D. thesis~\cite{ilb17} 
 at Middle East Technical University. We learned at the ICM 2018 satellite conference 
 `Braid Groups, Configuration Spaces and Homotopy Theory' held in Salvador, Brazil, that 
 Blazej Szepietowski has independently obtained similar results.
 The authors would like to thank Szepietowski for pointing out a mistake in the
  earlier version. The first author would like to thank to Mehmetcik Pamuk for helpful conversations.
  The second author thanks to UMASS Amherst Math. Department for its
  hospitality, where the writing of this paper is completed.

\section{Preliminaries}  \label{prelim}

In this section, we give the basic definitions and some facts on curve complexes. 
We introduce our model representing a non-orientable surface 
and the notation of the simple closed curves on it that will be used throughout the paper. 

For the isotopy classes $\al$ and $\be$ of two simple closed curves on a surface,
the \textit{geometric intersection number} of $\al$ and $\be$ is defined as
 \[
    i(\al,\be)=\min \{\mbox{Card}(a\cap b) \ |\ a\in \al, b\in \be \},
 \] 
  where $\mbox{Card}(a\cap b)$ denotes the cardinality of  $a\cap b$.  We say that
  $\al$ and $\be$ are \textit{disjoint} if $i(\al,\be)=0$.
  
 Throughout the paper, we will use the terms `hole' and `boundary component' to mean the same thing.
For simplicity, we will often confuse a curve and its isotopy class. 
It will be clear from the context which one is meant.   Similarly, a homeomorphism and its isotopy class
will be confused.

\subsection{Abstract simplicial complex} \ 

 An abstract simplicial complex on a set (of vertices) $V$ is a subset $K$ of
 the set of all subsets of $V$ satisfying the followings:
 \begin{itemize}
    \item[$\bullet$] $K$ does not contain the empty set $\emptyset$, 
    \item[$\bullet$] If $v\in V$, then $\{ v\} \in K$,
    \item[$\bullet$] If $\sigma\in K$ and if $\tau\subset \sigma$ with $\tau\neq \emptyset$, 
    		then $\tau\in K$.
  \end{itemize}
 An element of $K$ is called a \textit{simplex} of $K$.  A \textit {face} of a simplex 
 $\sigma$ in $K$ is a nonempty subset of $\sigma$.
 The \textit{dimension} of a simplex $\sigma$ of $K$ is defined as $|\sigma|-1$, and
 the \textit{dimension} of $K$ is the supremum of the dimensions of all simplices in $K$. 
 An element $v\in V$, identified with the $0$-simplex $\{ v\}$, 
 is called a  \textit{vertex} of $K$. An \textit{edge} 
 in $K$ is a $1$-simplex. If $\{ v_1,v_2\}$ is an edge, 
 we say that $v_1$ and $v_2$ are connected by an edge.
 
 A \textit{subcomplex} $L$ of a simplicial complex $K$ is a simplicial complex on a subset of $V$
 such that if $\sigma\in L$ then $\sigma\in K$. A subcomplex $L$ of $K$ is called a \textit{full subcomplex}
 if whenever a set of vertices in $L$ is a simplex in $K$, then it is also a simplex in $L$.
 
 For a vertex $v$ of $K$, the \textit{link}, $\link_K(v)$, of $v$ is the full subcomplex of $K$ on the set
 of vertices connected $v$ by an edge, and  the \textit{star} of $v$ is the full subcomplex
 containing $v$ and $\link_K(v)$, Note that  $\link_K(v)$ is a subcomplex of the star of $v$.

\subsection{Curve Complex $\CC(S)$} \ 
The curve complex $\CC(S)$ on a compact connected orientable surface $S$, introduced by Harvey~\cite{harv81},
is the abstract simplicial complex whose vertices are isotopy classes of nontrivial simple closed curves on $S$, and
a set of vertices $\{\al_0,\al_1,\ldots,\al_q\}$ forms a $q$-simplex if $i(\al_j,\al_k)=0$ for all $j,k$. 
A simple closed curve $a$ on $S$ is \textit{nontrivial} if it does not bound  a disc with at most one hole.

By an easy Euler characteristic argument, it is easy to see that all maximal simplices of $\CC(S)$
contains $3g+n-3$ elements (except for a few sporadic cases), 
so that the dimension of $\CC(S)$ is $3g+n-4$. Here, $g$ is the 
genus and $n$ is the number of boundary  components of $S$.

\subsection{Curve Complex $\CC(N_{g,n})$} \ 

Let $N=N_{g,n}$ be a compact connected non-orientable surface of genus $g$ with $n$ 
holes. 
The \textit{genus} of a non-orientable surface is defined as the maximum number 
of real projective planes in a connected sum decomposition. Equivalently, it is the maximum number 
of disjoint simple closed curves on $N$ whose complement is connected. Clearly, this definition of the
genus works for connected orientable surfaces as well.
By convention, 
we assume that a sphere is a non-orientable surface of genus $0$. 

Let $a$ be a simple closed curve on $N$. 
We denote by $N_a$ the surface obtained 
by cutting $N$ along $a$. The curve $a$ is called \textit{nonseparating} 
(respectively \textit{separating}) if $N_a$ is connected (respectively disconnected), and
 is called \textit{nontrivial} 
if it bounds neither a disc at most one hole nor a M\"{o}bius band.  
The curve $a$ is called \textit{two-sided} (respectively \textit{one-sided}) 
if a regular neighborhood of it is an annulus (respectively a M\"{o}bius band). 
When $a$ is one-sided, $a$ is called \textit{essential} if either $g=1$, or $g\geq 2$ and $N_a$ is 
non-orientable. If $a$ is separating,  then we  say that it is of \textit{type $(p,q)$} 
if $N_a$ is the disjoint union of two non-orientable surfaces one of which 
is of genus $p$ with $q+1$ holes. 
In particular, if $N_a$ is the disjoint union of a sphere with $ q+1$ holes and a non-orientable
surface of genus $g$ with  $n-q+1$ holes, that is 
if $a$ bounds a disk with $q$ holes, 
then we say that $a$ is of type $(0,q)$. Note that a simple closed curve of type $(p,q)$ 
is also of type $(g-p,n-q)$. We usually assume that $p\leq g-p$.

	The \textit{curve complex} $\CC(N)$ of the non-orientable surface $N$ 
	is is defined similarly to the orientable case; it is 
	an abstract simplicial complex with 
	vertices
	\[\{\al\ |\ \al \ \textrm{is the isotopy class of a nontrivial simple closed curve on} \ N \}.
	\]
   A  $q$-simplex of $\CC(N)$ is defined as a set of $(q+1)$ distinct vertices  
   $\{\al_0,\al_1,\ldots, \al_q\}$ such that $i(\al_j,\al_k)=0$ for all $j\neq k$.
   We note that $i(\alpha,\alpha)=1$ if $\alpha$ is one-sided, so that we require $j\neq k$.

 We finish this section with the following result stating  that the natural simplicial action of $\Mod(N_{g,n})$ 
 on $\CC(N_{g,n})$ induces an isomorphism:

\begin{theorem}{\cite[Theorem 1]{atako14}}\label{mod=aut}
	For $g+n\geq 5$, the natural map $\Mod(N_{g,n})\to \Aut(\CC(N_{g,n}))$ is an isomorphism.
\end{theorem}

\smallskip
\subsection{Top Dimensional Maximal Simplices} \
A set $P$ consisting of pairwise disjoint, nonisotopic nontrivial simple closed curves on $N$
 is called a \textit{pair of pants decomposition} of $N$, or simply a
  \textit{pants decomposition}, if each component of the surface obtained by cutting $N$
  along the curves in $P$ is a pair of pants, i.e. a three-holed sphere.
 For a connected surface of negative Euler characteristic, there is a bijection between maximal 
 simplices of the curve complex and 
 pants decompositions of the surface. In the case the surface is orientable 
 of genus $g$ with $n$ holes,  
 all maximal simplices of the curve complex have the same dimension $3g+n-4$.
 But, in the non-orientable case, dimensions of the maximal simplices 
 are  not constant when the genus is at least $2$.

\begin{lemma}\cite[Proposition 2.3]{atako14} \label{numberofcurve} 
	Let $N$ be a connected non-orientable surface of genus $g\geq 2$ with $n$ holes. 
	Suppose that $(g,n) \neq (2,0) $.  Let $a_r=3r+n-2$ and $b_r=4r+n-2$ if $g=2r+1$, and 
	$a_r=3r+n-4$ and $b_r=4r+n-4$ if $g=2r$. 
	Then, there is a maximal simplex of dimension $q$ in $\CC(N) $ 
	if and only if $a_r \leq q \leq b_r $. 
	In particular, there are precisely $\lceil{g/2}\rceil$ values 
	which occur as the dimension of a maximal simplex, 
	where $\lceil{g/2}\rceil$ denotes the smallest integer greater than $g/2$.
\end{lemma}

Note that the lemma holds true for non-orientable surfaces of genus $0$ and $1$ as well:
if $g=0$ then $r=0$ and $a_0=b_0=n-4$, and if $g=1$ then $r=0$ and $a_0=b_0=n-2$.

 We say that a pants decomposition $P$ on a non-orientable surface 
 is \textit{top dimensional} if its dimension is equal to $b_r$, so
 that it contains $b_r+1$ elements. 
 A top dimensional pant decomposition contains exactly $g$ essential one-sided curves 
 (cf. the proof of
 Lemma~\ref{numberofcurve} above given in ~\cite{atako14}).
 By Lemma~\ref{numberofcurve}, we have the following corollary. 

\begin{corollary}\label{toppants}
        Suppose that $g+n\geq 3$ and that $P$ is a top dimensional pants 
        decomposition of $N$. Then 
	$P$ contains $2g+n-3$ simple closed curves, exactly $g$ of which are 
	essential one-sided.  
\end{corollary}

\begin{lemma}\label{curvetype}
	Let $g\geq 1$, $P$ be a top dimensional pants decomposition of $N$ and $a\in P$. 
	Then the curve  $a$ is either essential one-sided, 
	or of type $(p,q)$ for some $p,q$ with $0\leq p\leq g, 1\leq q \leq n-1$.
\end{lemma}
\begin{proof}
       The lemma is trivially true for $g=1$; so suppose that $g\geq 2$. Recall that a curve 
       is of type $(p,q)$ if the curve is separating such that both components of its complement 
       are non-orientable and one of these components is of genus $p$ with $q+1$ holes.
       
	It suffices to prove that the pants decomposition $P$ 
	can not contain a simple closed curve of type that is not mentioned. 
	The topological types of simple closed curves 
	that are not mentioned in the lemma are as follows:
	\begin{enumerate}
		\item [$\bullet$] A two-sided nonseparating simple closed curve whose complement is 
		orientable (in the case $g$ is even).
		\item [$\bullet$] A one-sided simple closed curve whose complement is 
		orientable (in the case $g$ is odd).
		\item  [$\bullet$] A separating simple closed curve whose complement 
		is the disjoint union of an orientable surface of genus $k$ and 
		a non-orientable surface of genus $g-2k$ for some $k \geq 1$.
		\item  [$\bullet$] A two-sided nonseparating simple closed curve whose complement is non-orientable.
	\end{enumerate}

 Case 1. Let $\be_1$ be a two-sided nonseparating simple closed curve  
 such that the surface $N_{\be_1}$ is orientable and let $P_1$ be a pants decomposition of 
 $N$ containing $\be_1$, so that $Q_1=P_1\setminus \{ \beta_1\}$ is a pants decomposition
 of $N_{\beta_1}$. In this case,
 $g$ is even, say $g=2r$, $r\geq 1$. Since $N_{\be_1}$ is an orientable surface of
 genus $r-1$ with $n+2$ holes, the number of elements of $Q_1$ is $|Q_1|=3(r-1)+(n+2)-3$,
 so that $|P_1|=3r+n-3$. This number is less than the number $2g+n-3$. Hence, $P_1$ is not top dimensional.

 Case 2. Let $\be_2$ be a one-sided simple closed curve such that $N_{\be_2}$ is orientable 
 and let $P_2$ be a pants decomposition of $N$ containing $\be_2$. In this case, $g$ is odd, 
 say $g=2r+1$, and $Q_2\setminus \{ \beta_2 \}$ is a pants decomposition of $N_{\be_2}$,
 an orientable surface of genus $r$ with $n+1$ holes. 
 Thus,  $|Q_2|= 3r+(n+1)-3$, so that $|P_2|= 3r+n-1$, which is less than $2g+n-3$.
 Hence, $P_2$ is not top dimensional.
 
 Case 3. Let $\be_3$ be a separating simple closed curve such that 
 $N_{\be_3}$ is the disjoint union of an orientable surface $S$ of genus $k\geq 1$ 
 with $l+1$ holes, and  a non-orientable surface $N'$
 of genus $g-2k$ with $n-l+1$ holes. 
 Let $P_3$ be a pants decomposition of $N$ containing $\be_3$. The
  elements of $P_3$ lying in the interior of $S$ is a pants decomposition $Q_3$ of $S$, 
  and the elements
 lying in the interior of $N'$ is a pants decomposition $Q'_3$ of $N'$.
 Then, $P_3=Q_3\cup Q'_3\cup \{ \beta_3 \}$, $|Q_3|=3k+(l+1)-3$, 
  $|Q'_3|\leq 2(g-2k)+(n-l+1)-3 $. It follows that 
  $|P_3|= |Q_1|+|Q_2|+1 \leq 2g+n-3-k$, which is less than $2g+n-3$, so that
 $P_3$ is not top dimensional. 
 
 Case 4. Let $\be_4$ be a two-sided nonseparating simple closed curve 
 such that $N_{\be_4}$ is non-orientable. Let $P_4$ be a pants decomposition of $N$ 
 containing $\be_4$. Since $N_{\be_4}$ is a non-orientable surface of genus $g-2$ with $n+2$ holes, the 
 pants decomposition $Q_4 =P_1\setminus\{\be_4\}$ contains at most 
 $2(g-2)+(n+2)-3$ elements, so that
 $|P_4|= |Q_4|+1\leq 2g+n-4$, concluding that $P_4$ cannot be top dimensional.

 This finishes the proof of the lemma.
 \end{proof}

 \smallskip

 \subsection{Our model for $\Ngn$ and curves on it} \label{sec:model} \ 

 Let $R$ be a disc in the plane whose boundary is a $(2g+n)$-gon
 such that the edges are labeled as $s_1, e_1,s_2,e_2,\ldots,s_g,e_g,e_{g+1},e_{g+2}, \ldots,e_{g+n}$  
 in this order.  We denote the edges $s_i$ by boldface semicircles
 (see Figure~\ref{model}).  For $k=1,2,\ldots,n$, let $z_k$ be the corner between the edges
 $e_{g+k-1}$ and $e_{g+k}$. We glue two copies of $R$ along $e_i$ so that each point on $e_i$ is identified with
 its copy, so that  the result is a sphere $S$ with $g$ holes.
 By identifying the antipodal points on each boundary component of $S$    
 and  by deleting a small open disc about each $z_1,z_2,\ldots,z_n$, we obtain
 a non-orientable surface $N_{g,n}$ of genus $g$ with $n$ punctures.
 The disc $R$ will be our model for $N_{g,n}$, so that we understand that $N_{g,n}$ is obtained from 
 the double of
 $R$ in the way we just described. 

\begin{figure}[hbt!]
\labellist
\small\hair 2pt
 \pinlabel $s_1$ at 107 550
 \pinlabel $s_2$ at 62 455
 \pinlabel $s_g$ at 105 360
 \pinlabel $e_1$ at 100 500
 \pinlabel $e_2$ at 88 427
 \pinlabel $e_{g-1}$ at 109 396
 \pinlabel $e_g$ at 165 370
 \pinlabel $e_{g+1}$ at 218 390
 \pinlabel $e_{g+2}$ at 246 430
  \pinlabel $e_{g+n-2}$ at 236 480
  \pinlabel $e_{g+n-1}$ at 220 510
  \pinlabel $e_{g+n}$ at 170 530
  \pinlabel $z_1$ at 220 350
 \pinlabel $z_2$ at 277 410
 \pinlabel $z_{n-1}$ at 285 495
 \pinlabel $z_n$ at 232 547

  \endlabellist
	\centering
	\includegraphics[width=0.8\textwidth]{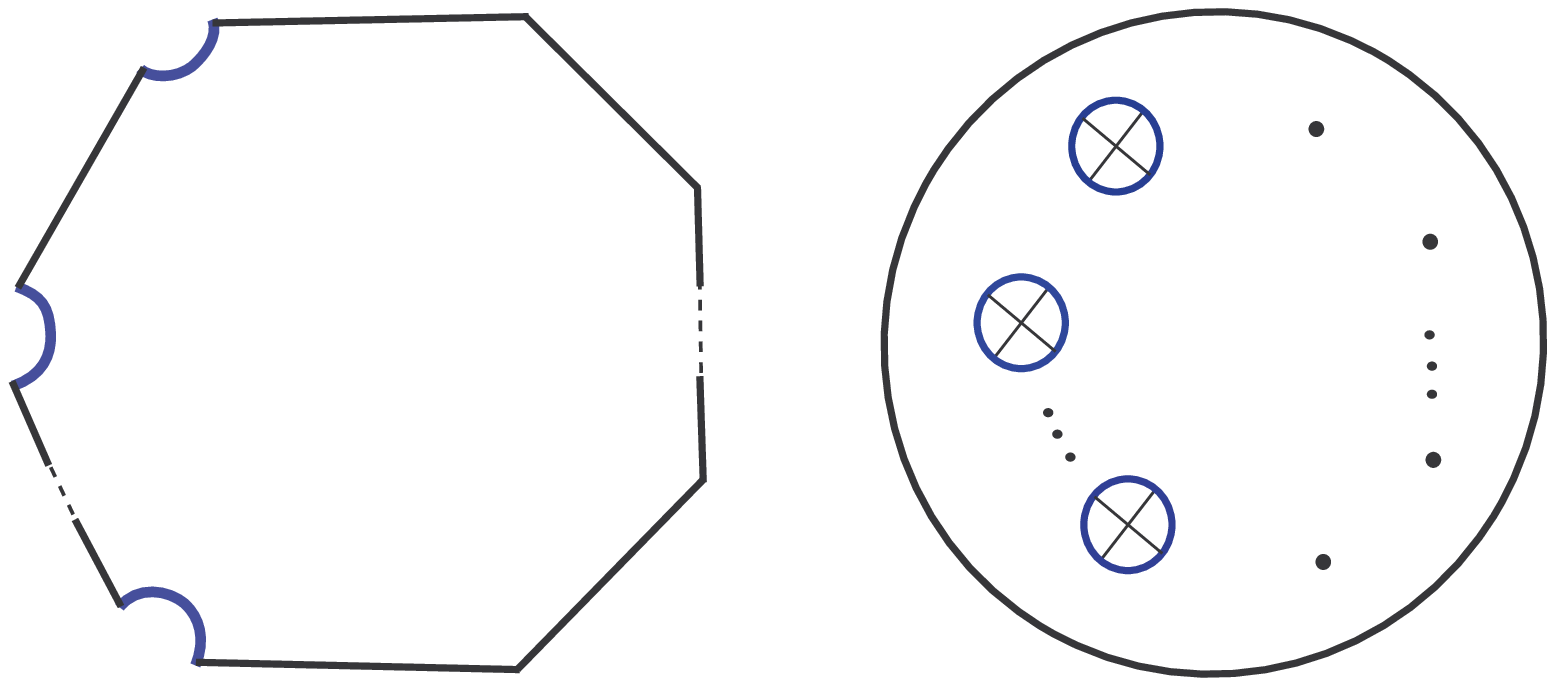}
	\caption{The disc $R$ on the left represents the surface $N_{g,n}$ on the right, a holed-sphere
	with discs having crosses, where 
	the interiors of the discs 
	with a cross are to be deleted and the antipodal points on each resulting boundary components 
	are to be identified.} \label{model}
\end{figure}

 A properly embedded arc on the model $R$ whose endpoints are on the boundary of 
 $R$ (but different from $z_j$) gives a simple closed curve on $N_{g,n}$. Curves we are
 interested in are represented by such arcs on $R$. We fix the 
 following curves on $N_{g,n}$ (cf. Figure~\ref{arcs}):
 \begin{enumerate}
   \item[$\bullet$]  For $1\leq i\leq g$, let $\alpha_i$ be the one-sided 
   										curve represented by $s_i$. 
   \item[$\bullet$]  For $1\leq i\leq g$, $1\leq j\leq g+n$ with $j \neq i, j\neq i-1$ (mod $(g+n)$), 
  		 			let $\alpha_i^j$ be the one-sided curve represented by
   					a line segment  joining the midpoint of $s_i$ and a point of $e_j$.
   \item[$\bullet$] For $1\leq i,j\leq g+n$ and $|i- j|\geq 2$, 
   			let $\beta_i^j$ be the two-sided curve represented by 
   			a line segment joining a point of $e_i$ to a point of $e_j$. 
 \end{enumerate}

 We consider isotopic copies of these curves in such a way that they are in minimal position. 
 Note that $\beta_i^j=\beta_j^i$.

\begin{figure}[hbt!]
\labellist
\small\hair 2pt
 \pinlabel $\alpha_1$ at 107 550
 \pinlabel $\alpha_i$ at 60 455
 \pinlabel $\alpha_g$ at 102 357
 \pinlabel $\alpha_1^j$ at 180 520
 \pinlabel $\alpha_i^j$ at 150 466
 \pinlabel $\alpha_g^j$ at 170 415
 \pinlabel $e_g$ at 170 346
 \pinlabel $e_j$ at 278 455
 \pinlabel $e_{g+n}$ at 170 555 
 \pinlabel $\beta_g^j$ at 430 385
 \pinlabel $\beta_{g-1}^j$ at 380 412
 \pinlabel $\beta_i^j$ at 375 450
 \pinlabel $\beta_1^j$ at 390 488
 \pinlabel $\beta_{g+n}^j$ at 430 525
    \pinlabel $\alpha_1$ at 360 550
 \pinlabel $\alpha_i$ at 313 455
 \pinlabel $\alpha_g$ at 355 357
 \pinlabel $e_g$ at 423 346
 \pinlabel $e_j$ at 531 455
 \pinlabel $e_{g+n}$ at 413 555
   \endlabellist
	\centering
	\centering
	\includegraphics[width=0.7\textwidth]{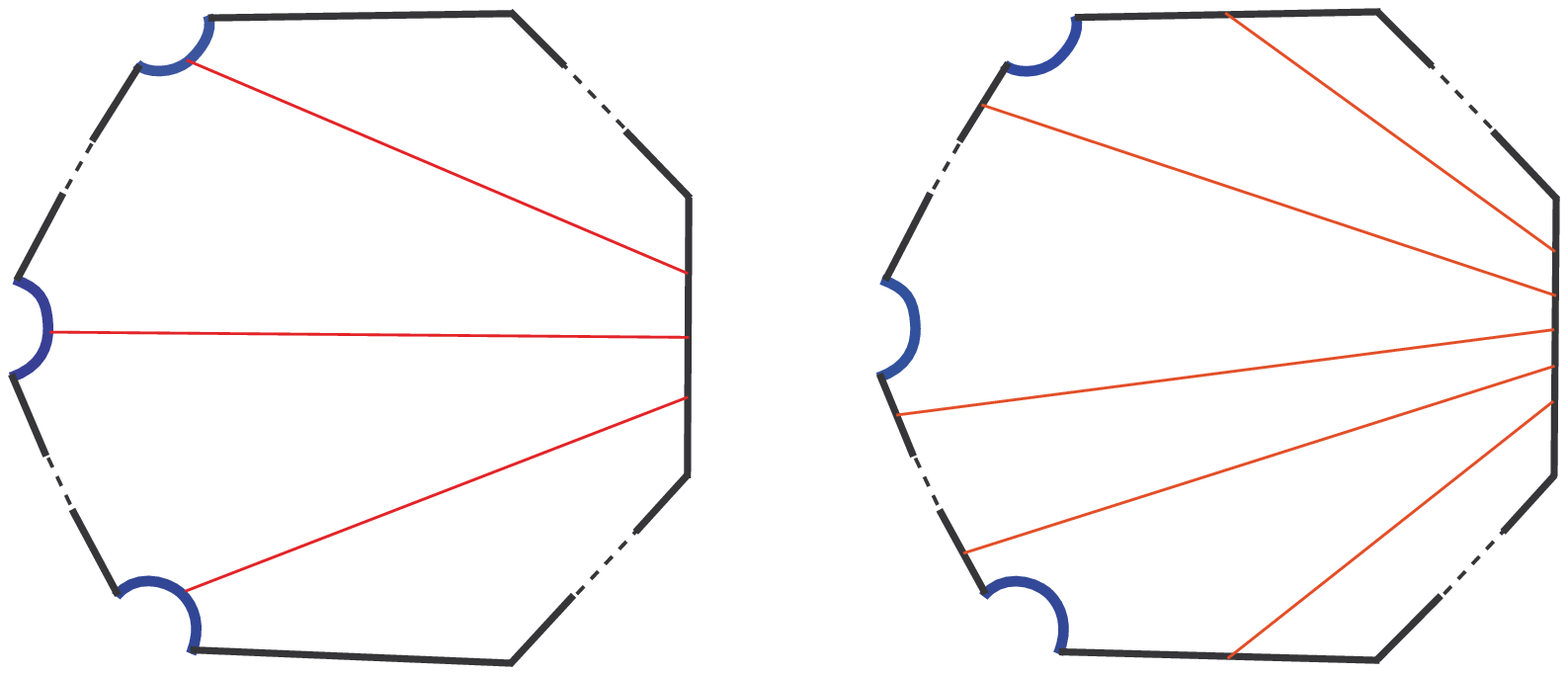}
	\caption{Arcs representing simple closed curves on $N_{g,n}$.}\label{arcs}
\end{figure}

\subsection{Locally Injective Simplicial Maps} \
 Let $K$ be an abstract simplicial complex and $\ga$ be a vertex in $K$. 
 Recall that the \textit{star} of  $\ga$ is the subcomplex of $K$ consisting of 
 all simplices of $K$ containing $\ga$ as a vertex and all faces of such simplices.
For a subcomplex $L$ of $K$,
 a simplicial map $\phi:L\rightarrow K$ is called \textit{locally injective} if its restriction 
 to the star of every vertex of $L$ is injective. 

 \begin{remark} 
  Note that if $L$ is a subcomplex of $\CC(N)$, $\sigma$ is a simplex in $L$ and if $\phi:L\to \CC(N)$
  is locally injective,  then the dimension of $\phi(\sigma)$ is equal to that of $\sigma$. In particular, if
  $\sigma$ is a top dimensional pants decomposition then so is $\phi(\sigma)$.
\end{remark}

\subsection{Adjacency Graph of a Pants Decomposition} \
      Let $P$ be a pants decomposition of $N$. Two elements $\al$ and $\be$  of $P$ are  
 	called \textit{adjacent} with respect to $P$ if there
	exists a connected component, a pair of pants, of $N_P$ (the surface 
	obtained by cutting $N$ along all elements of $P$) 
	 containing $\al$ and $\be$ as its boundary components.  The \textit{adjacency graph} $\gap$ of $P$ 
	is the graph whose vertices are the elements of $P$;  
	two vertices $\al$ and $\be$ are connected by an edge if  
	$\al$ and $\be$ are adjacent with respect to $P$ (cf. \cite{behrmar06}). 
	 We say $P$ is \textit{linear} if $\gap$ is linear.  
	If two vertices are adjacent (recpectively, nonadjacent) with respect to $P$, 
	we will also say that they are adjacent (respectively, nonadjacent) in 
	 $\gap$.
 
\begin{figure}[hbt!]
\labellist
\small\hair 2pt 
	\pinlabel $\alpha_1$ at 458 550
	\pinlabel $\beta_{1}^6$ at 458 510
	\pinlabel $\beta_{1}^5$ at 458 470
	\pinlabel $\beta_{1}^4$ at 458 430
	\pinlabel $\beta_{1}^3$ at 458 390
 	\pinlabel $\alpha_2$ at 390 352
 	\pinlabel $\alpha_3$ at 487 352
	 \pinlabel $\alpha_1$ at 175 540
 	\pinlabel $\alpha_2$ at 125 450 
  	\pinlabel $\alpha_3$ at 168 354
  	\pinlabel $\beta_{1}^6$ at 262 500
	\pinlabel $\beta_{1}^5$ at 285 475
	\pinlabel $\beta_{1}^4$ at 265 430
	\pinlabel $\beta_{1}^3$ at 228 390  
  \endlabellist
	\centering
	\centering
	\includegraphics[width=0.55\textwidth]{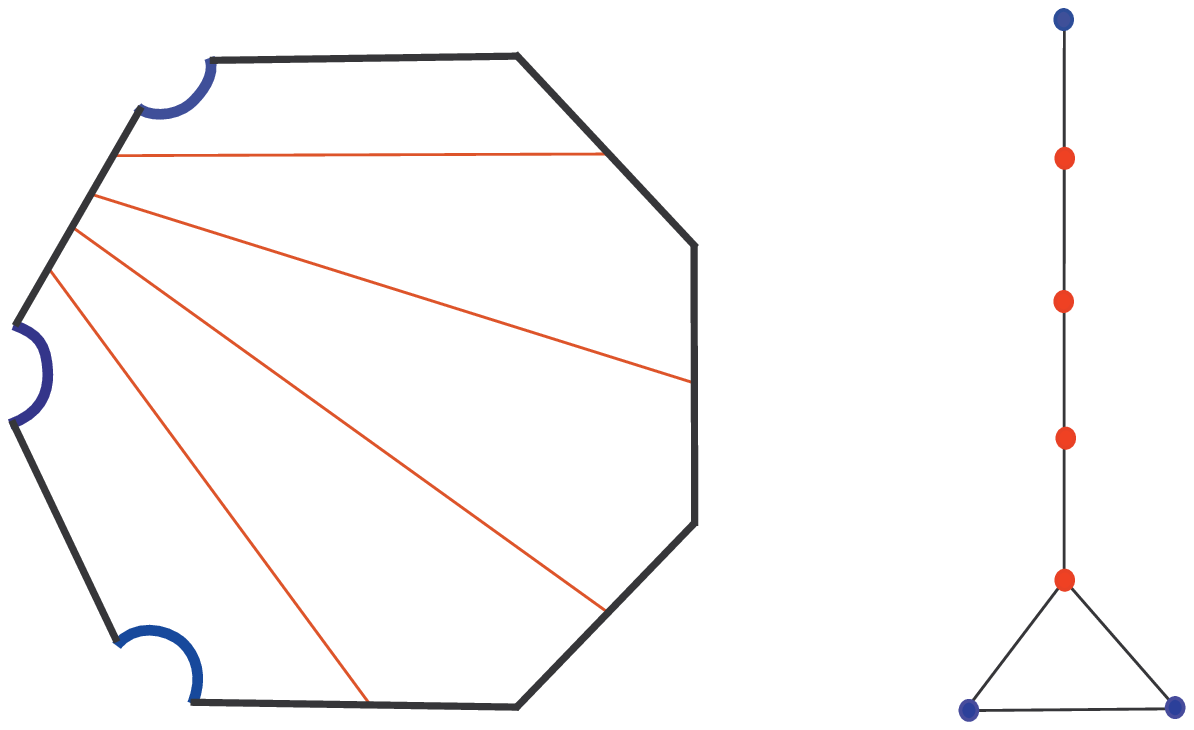}
	\caption{Simple closed curves in a pants decomposition of 
	$N_{3,4}$ and its adjacency graph.}\label{3-4adj}
\end{figure}

 Let $\al \in \gap$ be a vertex of valency one.  A vertex $\beta$ is called 
 a \textit{$k$th linear successor} of $\alpha$ if there is a path 
$\alpha=\gamma_0,\gamma_1,\ldots,\gamma_{k-1},\gamma_k=\beta$ from $\alpha$ to 
 $\beta$ such that each $\gamma_i$, $1\leq i\leq k-1$, 
 has valency two in $\gap$.
 In Figure~\ref{3-4adj}, the vertex $\be_1^6$ is a first linear successor, $\be_1^5$ is 
 a second linear successor and  $\be_1^3$ is a $4$th linear successor of $\al_1$.

\section{The Finite Subcomplex $\protect\kgn$ of The Curve Complex}\label{sec:subcomplex} 
A subcomplex $L$ of the curve complex $\CC(N)$ of the non-orientable surface $N$ is called \textit{rigid} 
if every locally injective simplicial map  $L\to \CC(N)$
is induced from a homeomorphism $N\to N$. By Theorem~\ref{mod=aut}, the curve complex itself is rigid
if $g+n\geq 5$.

Let $N_{g,n}$ be the compact connected non-orientable surface of genus $g$ 
with $n$ holes whose model is given in Section~\ref{sec:model}.
In this section, we introduce the finite subcomplex $\kai_{g,n}\subset\CC(N_{g,n})$, 
 which
will be shown to be rigid in $\CC(N_{g,n})$.

\subsection{The Finite Subcomplex $\kai_{g,n}$}\label{sec:kai}

 We use the curves introduced in 
Section~\ref{sec:model} to define the finite subcomplex $\kaign$ of $\CC(N_{g,n})$. 
We define two sets of vertices in $\CC(N_{g,n})$ as
\begin{eqnarray*}
	\kai_{g,n}^1 &=& \{ \alpha_i,\alpha_i^j\ \mid \ 1\leq i\leq g,\ 1 \leq j \leq g+n \ 
								\textrm{and}  \ j\neq i, j\neq i-1 \mbox{ mod }(g+n) \}, \\
	\kai_{g,n}^2 &=& \{\beta_i^j \ | \ 1 \leq i,j \leq g+n,\ 2 \leq |i-j| \leq  g+n-2 \}.
\end{eqnarray*}
We set $\kai_{g,n}=\kai_{g,n}^1\cup \kai_{g,n}^2$ (cf. Figure~\ref{g-nkai}). 
Although we write $\kai_{g,n}^1, \kai_{g,n}^2$ and $\kai_{g,n}$ as sets of vertices, 
we understand the full subcomplexes of $\CC(N)$ containing these sets.
Note that the vertices in $\kai_{g,n}^1$ are one-sided essential, and that 
 the vertices in $\kai_{g,n}^2$ are separating. 
 
 We remark that cutting $N_{g,n}$ along $\alpha_g$ induces an isomorphism
    \[
   \varphi : \link_{\CC(N_{g,n})}(\alpha_g) \to \CC(N_{g-1,n+1}) 
 \]
 mapping the link $\link_{\kai_{g,n}}(\alpha_g)$ of $\alpha_g$  to  
 $\kai_{g-1,n+1}$ isomorphically.
 This allows us to make induction on the genus of 
 the surface in the proof of the main theorem.

\begin{figure}[hbt!]
\labellist
\small\hair 2pt   
    \endlabellist
	\centering
		\centering
		\includegraphics[width=0.75\textwidth]{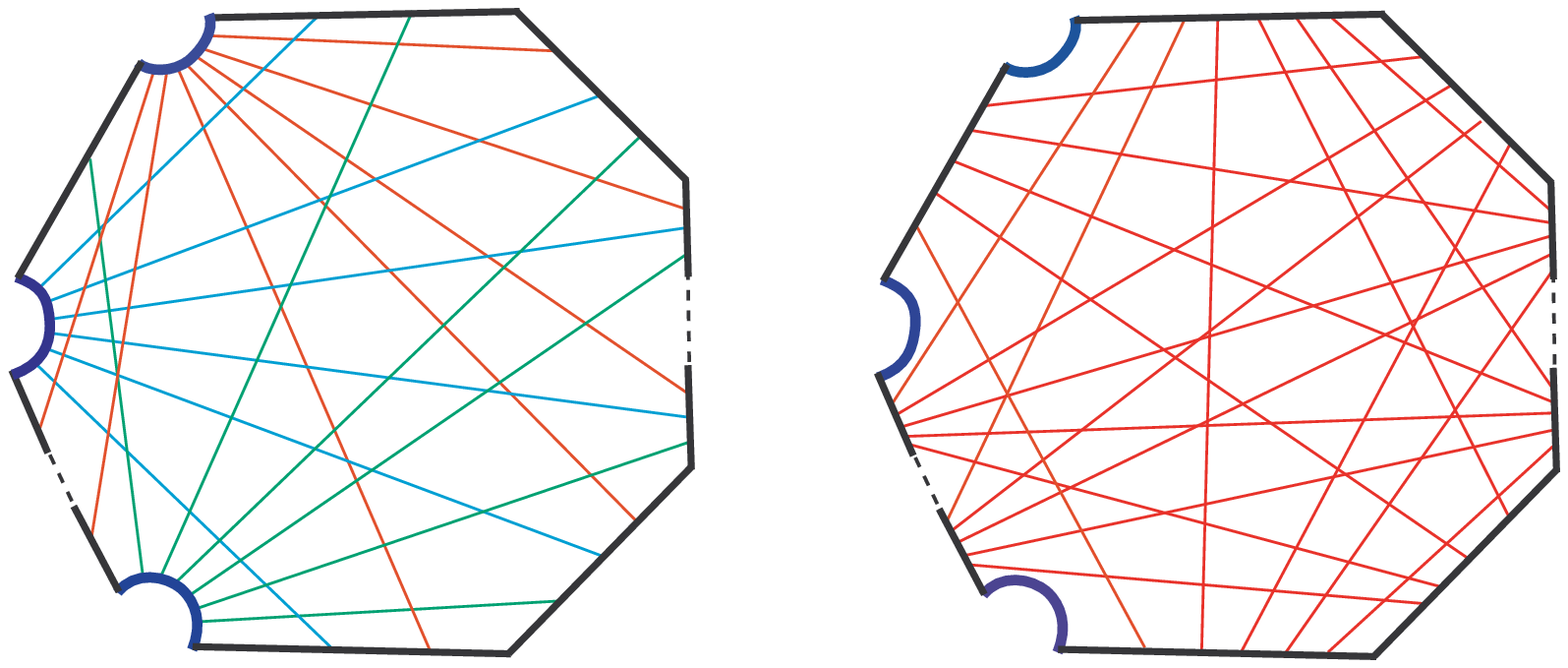}
		\caption{Vertices of $\kai_{g,n}^1$ and $\kai_{g,n}^2$.}\label{g-nkai}
\end{figure}

We note that for $n\geq 5$, the complex $\kai_{0,n}$ is the complex $\kai$ of Aramayona-Leininger. 
We may restate Theorem 3.1 in~\cite{aramlei13}
as follows.

\begin{thm} [Aramayona-Leininger] \label{thm:arle}
For $n\geq 5$, any locally injective simplicial map $\phi:\kai_{0,n}\to \CC(N_{0,n})$
is induced by a unique orientation preserving homeomorphism $N_{0,n}\to N_{0,n}$.
\end{thm}

\subsection{The nonadjacency condition.}
 
 A top dimensional pants decomposition $P$ of $N_{g,n}$ contained
 in a subcomplex $L$ of the curve complex \ $\CC(\Ngn)$ is said to satisfy
 the  \textit{nonadjacency condition} with respect to $L$ if 
  for every pair $\al$ and $\be$ of nonadjacent vertices  in $P$,  
	there exist two disjoint simple closed curves $\dela$ and $\delb$ in $L$ different from
	$\al$ and $\be$, respectively,
	such that 
	\[
			i(\dela,\al) \neq 0,   i(\delb, \be) \neq 0 \ \textrm{and} \ 
			i(\dela, \be) = 0, i(\delb, \al) = 0
	\] 
	and each of them is disjoint from the simple closed curves 
	in $P\setminus \{\al,\be\}$, 
	i.e. $\left( P  \setminus \{\al \} \right) \cup \{ \delb\}$,
	$\left( P  \setminus \{\be \} \right) \cup \{ \dela\}$ and
	$\left( P  \setminus \{\al,\be \} \right) \cup \{ \dela,\delb\}$ are
	pants decompositions.
	If $L=\kai_{g,n}$ we simply say that $P$ satisfies the 
	nonadjacency condition.

The following lemma is the non-orientable version of~\cite [Lemma 7] {shack07}, 
which states that for a compact orientable surface $S$ 
any locally injective simplicial map 
$\CC(S)\rightarrow\CC(S)$ preserves the nonadjacency in the adjacency graph of 
a pants decomposition of $S$. For the sake of completeness, we give a proof of the lemma. 

\begin{lemma}\label{nonadjacency}
	Let $L$ be a subcomplex of \ $\CC(\Ngn)$, $P$ be a top dimensional 
	pants decomposition of $N_{g,n}$ contained in $L$
	and let $\phi :L\rightarrow\CC(\Ngn)$ be a locally injective simplicial map.
	Suppose that $P$ satisfies the nonadjacency condition with respect to $L$. 
	If $\alpha$ and $\beta$ are nonadjacent in $P$, then $\phi(\alpha)$ 
	is nonadjacent to $\phi(\beta)$ in the pants decomposition $\phi(P)$.
\end{lemma}
\begin{proof}
	Let $\al$ and $\be$ be two nonadjacent simple closed curves in $P\subset L$.
	By assumption, there exist two disjoint simple closed curves 
	$\dela$ and $\delb$ on $N_{g,n}$, different from $\al$ and $\be$ respectively,
	such that $i(\dela , \al) \neq 0,   \ i(\delb, \be)\neq 0$,
	$\dela \cap \be = \emptyset, \delb \cap \al = \emptyset$ and 
	each of $\dela$ and $\delb$ is disjoint from the elements of $P\setminus \{\al,\be\}$.
	Hence, $\phi(\dela)$ and $\phi(\delb)$ are disjoint from each other and from all curves in the set
	$\phi(P)\setminus\{\phi(\al),\phi(\be)\}$. Also $\phi(\dela)$ is disjoint from 
	$\phi(\be)$ and $\phi(\delb)$ is disjoint from 
	$\phi(\al)$. It is then easy to see that $\phi(\delta_\alpha)$ 
	 must intersect $\phi(\alpha)$
	and $\phi(\delta_\beta)$  must intersect $\phi(\beta)$.
	
	Assume $\phi(\al)$ and $\phi(\be)$ are adjacent in $\phi(P)$. Then 
	there is a subsurface $X$ of $N_{g,n}$ homeomorphic to a pair of pants 
	bounded by $\phi(\al),\phi(\be)$ and a third curve (possibly a boundary component of $N$), and
	$\phi(\dela)$ and $\phi(\delb)$ do not intersect the third curve. Since each component of 
	$\phi(\dela)\cap X$ (respectively, of $\phi(\delb)\cap X$) is an arc connecting two points of $\phi(\al)$ 
	(respectively $\phi(\be)$) and since any two such arc on a pair of pants must intersect 
	(cf.~\cite{flp}), we conclude that $\phi(\dela)$ and $\phi(\delb)$ intersect each other. By this contradiction, 
	$\phi(\al)$ and $\phi(\be)$ are nonadjacent in $\phi(P)$.
 \end{proof}

\section{Preserving Topological Types of Vertices in $\kaign$}\label{sec:types}

Every self-homeomorphism of a surface preserves topological types of simple closed curves. 
Thus,  a map from a subcomplex to the curve complex must preserve topological types of vertices
of the subcomplex, in order to have a hope that it is induced by a homeomorphism.

The purpose of this section is to prove that the topological types 
of many vertices, enough to prove the main theorem, of $\kaign$  are preserved under locally 
injective simplicial maps $\kaign \rightarrow \cn$. To this end, let
$\phi:\kaign\rightarrow\cn$ be a locally injective simplicial map.
In this section we assume that $g+n\geq 5$.

\subsection {The case $g=1$}  \
Let $N=N_{1,n}$ denote the non-orientable surface of genus $1$ with $n\geq 4$ holes.
We show that $\phi$ preserves the topological types of vertices in $\kai_{1,n}$.
Note that in this case $\kai_{1,n}$ is the subcomplex of $\CC(N)$ with the vertex set
 \[
 \{ \alpha_1, \alpha_1^2, \alpha_1^3,\ldots, \alpha_1^n     \} \cup
 \{\beta_i^j \ | \ 1 \leq i,j \leq n+1,\ 2 \leq |i-j| \leq  n-1 \}.
 \]
By Corollary~\ref{toppants}, a pants decomposition of $N$
contains exactly $n-1$  curves. Note that every pants decomposition contains exactly 
$n-1$ curves.

\subsubsection {Linear Pants Decompositions}

Recall that a pants decomposition is called \textit{linear} 
if its adjacency graph is linear.

\begin{lemma}\label{cccc}
        For every curve $\beta_i^j\in \kai_{1,n}$ there is a linear pants decomposition  
        $P_i^j $ in  $\kai_{1,n}$ containing both $\alpha_1$ and $\beta_i^j$.
\end{lemma}
\begin{proof}
Clearly we can assume that $i\leq j-2$. If $i=1$ 
\[
P_1^j = \{  \alpha_1,\beta_1^ n,   \beta_1^{n-1},\ldots,\beta_1^{j+1},\beta_1^j,  \beta_1^{j-1},\ldots,  \beta_1^3   \},
\] 
and otherwise
\[
  P_i^j = \{  \alpha_1,\beta_2^ {n+1},   \beta_2^{n},\ldots,  \beta_2^{j+1}, \beta_2^j,\beta_3^j,  \beta_4^j, 
  \ldots,  \beta_{i-1}^j, \beta_i^j,    \beta_i^{j-1},  \beta_i^{j-2}, \ldots,  \beta_i^{i+2}   \}
\]
is a desired pants decomposition.
\end{proof}

\begin{lemma}\label{valency2}
	Let $P$ be a linear pants decomposition of $N$, 
	$\gap$ be its adjacency graph and  
	$\ga \in \gap$ be a vertex. 
	The vertex $\ga$ is of type $(0,k)$ for some $k\geq3$ if and only if 
	$\ga$ has valency two in $\gap$.
\end{lemma}
\begin{proof}
	Suppose first that $\ga$ is of type $(0,k)$, for some $k\geq3$. 
	The surface $N_\gamma$ has two connected components, 
	none of which is a pair of pants. 
	Thus, each connected component contains at least an element of $P$
	adjacent to $\gamma$, 
	so that $\gamma$ is adjacent to at least two, 
	hence exactly two, vertices in $\gap$ since $P$ is linear.
	
	Suppose now that $\ga$ has valency two in $\gap$.  Assume $\ga$ is one-sided. 
	If $N$ is cut along the elements of $P$, one gets $n-1$ pairs of pants, and 
	there is a unique pair of pants $Y$ such that 
	its boundary components come from $\gamma$ 
	and two other different simple closed curves, say $\be_1$ and $\be_2$,  in $P$. 
	But in this case the vertices $\gamma, \be_1,\be_2$ form a triangle in the adjacency graph, 
	contradicting the linearity of $\gap$. By a similar argument, 
	$\ga$ can not be of type $(0,2)$ either. 	
 \end{proof}

\begin{remark}\label{valency1}
	Note that a vertex in the adjacency graph of a linear pants decomposition  
	has exactly two vertices of valency one. These vertices are either one-sided or of type $(0,2)$. 
	In fact, one of them is one-sided and the other is of type $(0,2)$ since the genus of the surface is one.
\end{remark}

\begin{lemma}\label{lingraph}
	Let $P=\{\ga_2,\ga_3,\ldots ,\ga_n\}$ be a linear pants decomposition of $N$, 
	$\gap$ be the adjacency graph of $P$  such that 
	$\ga_i$ is connected to $\ga_{i+1}$ for $i=2,\ldots,n-1$ in $\gap$. 
	Then, up to changing the index $i$ with $n+2-i$ if necessary, 
	$\ga_n$ is one-sided and $\ga_i$ is of type $(0,i)$ for $2\leq i \leq n-1$.
\end{lemma}
\begin{proof}
	By Remark~\ref{valency1}, either $\gamma_2$ or $\gamma_n$ is one-sided.
	By changing the labeling if necessary, we may assume that $\ga_n$ is one-sided. 
	It remains to prove that $\ga_k$ is of type $(0,k)$ for $2\leq k\leq n-1$.
		
	The proof of this is by induction on $k$. 
	Since the vertex $\ga_2$ has valency one,
	it is of type $(0,2)$ by Remark~\ref{valency1}. 
	Assume that $\ga_i$ is of type $(0,i)$ for $i\leq k-1$. 
	If $N$ is cut along all curves in $P$, one gets $n-1$ pairs of pants.
	There is a unique pair of pants $X$ such that 
	two of the holes of $X$ are $\ga_{k-1}$ and $\ga_k$
	as they are adjacent with respect to $P$. 
	If the third hole of $X$ were a curve 
	$\be \in P\setminus\{\ga_{k-1},\ga_k\}$, then $\gap$ would 
	contain a triangle with vertices $\{\be,\ga_{k-1},\ga_{k}\}$, 
	contradicting with the linearity of $P$. 
	Hence, the third hole of $X$ is a hole of $N$
	as well.
	Since $\ga_{k-1}$ bounds a disc with $k-1$ holes, it follows that
	$\ga_k$ bounds a disc with $k$ holes, i.e. it is of type $(0,k)$.
\end{proof} 

 The following lemma is based on \cite[Lemma 5.2]{behrmar06} that proves 
 that if the adjacency graph of a pants decomposition of 
 a torus with $n\geq 4$ holes does not contain a triangle, 
 then this pants decomposition is either linear 
 or a cyclic pants decomposition.

\begin{lemma}\label{notriangle}
	Let $P$ be a pants decomposition of $N$ and 
	let $\gap$ be its adjacency graph. If $\gap$ does 
	not contain any triangle, then $P$ is linear.
\end{lemma}
\begin{proof}
	Suppose that $\gap$ does not contain any triangle. The surface $N_P$
	obtained by cutting $N$ along the elements of $P$
	is a disjoint union of pair of pants. Since $\gap$ does not 
	contain any triangle, at least one hole of 
	each pair of pants is a hole of $N$. If $X$ is a connected component of $N_P$, which is
	a pair of pants, then
	either	
	\begin{enumerate}
		\item  exactly two holes of $X$ are curves in $P$ 
					one of which is one-sided, or 
		\item  exactly two holes of $X$ are two-sided curves in $P$, or 
		\item  exactly one hole of $X$ is a curve in $P$, which must be two-sided. 
	\end{enumerate} 
	
	Since the genus of $N$ is $1$, there is only one pair of pants $X_1$ of type $(1)$. 
	The pair of pants $X_2$ glued to $X_1$ must be of type $(2)$; otherwise, i.e., if type $(3)$, 
	the surface would be $N_{1,3}$, which is impossible. The pair of pant $X_3$ glued to $X_1\cup X_2$ can be 
	of type $(2)$ or $(3)$. If type $(3)$, then $N=N_{1,4}$. By arguing in this way
	we conclude that $N=X_1\cup X_2\cup X_3\cup \cdots \cup X_{n-1}$, where
	$X_1$ is of type $(1)$, $X_{n-1}$ is of type $(3)$ and all others are of type $(2)$. 
	It follows that $\gap$ is linear.
\end{proof}

\begin{lemma}\label{lintolin}
	Let $P\subset \kai_{1,n}$ be a linear pants decomposition 
	of $N$ satisfying the nonadjacency condition with respect to $\kai_{1,n}$.
	Then $\phi(P)$ is a linear pants decomposition of $N$. 
\end{lemma}
\begin{proof}
	Suppose first that $n=4$. In this case, the graph $\gap$ contains 
	three vertices and two of them, say $\alpha,\beta$ 
	are nonadjacent to each other. If the adjacency graph 
	$\gapf$ of $\phi(P)$ is not linear, then it must be a triangle, i.e.
	each vertex has valency two. In particular,
	$\phi(\alpha)$ is adjacent to $\phi(\beta)$.
	Since $\phi$ preserves nonadjacency with respect to $P$ by Lemma~\ref{nonadjacency}, 
	this is impossible. 
	
	Suppose now that $n\geq 5$. Then $P$ contains at least four curves. Assume that $\gapf$ 
	contains a triangle. Since $\gapf$ is connected and contains at least four vertices, there 
	exists a vertex $\phi(\ga)$ in this triangle which has valency at least three. 
	Thus, 
	there exist at most $n-5$ vertices in $\gapf$ nonadjacent 
	to $\phi(\ga)$. On the other hand, 
	the number of the vertices in $\gap$ nonadjacent to $\ga$ is 
	either $n-3$ or $n-4$.
	Since $\phi$ preserves nonadjacency by Lemma~\ref{nonadjacency}, 
	this is a contradiction.
	 
	It follows now from Lemma~\ref{notriangle} that $\phi(P)$ is linear.	
\end{proof}

\begin{lemma}\label{lineer->nonadj}
	Let $P\subset \kai_{1,n}$ be a linear pants decomposition of 
	$N$ containing $\alpha_1$. Then $P$ satisfies the nonadjacency condition. 
\end{lemma}
\begin{proof}
We may assume, by Lemma~\ref{lingraph}, that the elements of the pants decomposition
$P$ are $\gamma_2,\gamma_3,\ldots,\gamma_{n-1},\gamma_n=\alpha_1$, where
$\gamma_k$ is of type $(0,k)$ for $2\leq k\leq n-1$. 
We claim that for every $\gamma\in P$, there is a vertex $\delta_{\gamma}$
in $\kai_{1,n}$ such that $\delta_\gamma\neq \gamma$ and 
$(P\setminus\{ \gamma\})\cup \{ \delta_\gamma\}$ is also a pants decomposition.
Since for any two nonadjacent vertices $\al$ and $\be$ 
there is a separating vertex $\delta\in P$ such that $\al$ and $\be$ lie on different 
connected components of $N_\delta$, the lemma follows from this claim, 

\begin{table}[ht]
\begin{center}
\centering
\begin{tabular}{ ||c|c|c|c|c|| } 
 \hline
	$k$ & $\ \ \gamma=\gamma_k $ \ \ 
	& \ \ $\gamma_{k-1 }$ \ \  & \ \ $\gamma_{k+1 }$ \ \ 	
	&\ \  $\delta_\gamma$ \ \ 
\\   \hline \hline
	 \multirow{2}{*}{$2$} &  \multirow{2}{*}{$\beta_{j}^{j+2}$}
	 & \multirow{2}{*}{---}
	  & $\beta_{j}^{j+3}$ & $\beta_{j+1}^{j+3}$ 
\\  \cline{4-5} 
	& & &$\beta_{j-1}^{j+2}$& $\beta_{j-1}^{j+1}$
 \\  \hline
	\multirow{4}{*}{$2<i<n-1$} 
	& \multirow{4}{*}{$\beta_{i}^{j}$} 
 	& \multirow{2}{*}{$\beta_{i}^{j-1}$} 
	 & $\beta_{i}^{j+1}$  
	 &  $\beta_{j-1}^{j+1}$
\\   \cline{4-5} 
	&& & $\beta_{i-1}^{j}$  & $\beta_{i-1}^{j-1}$
\\ \cline{3-5}
	&&  \multirow{2}{*}{$\beta_{i+1}^{j}$} 
	&  $\beta_{i}^{j+1}$   
 	 & $\beta_{i+1}^{j+1}$ 
\\   \cline{4-5}
	&&& $\beta_{i-1}^{j}  $ & $\beta_{i-1}^{i+1}$ 
  \\          
 \hline 
	\multirow{4}{*}{$n-1$}&\multirow{2}{*}{$\beta_{1}^{n}$} 
 	 & $\beta_{2}^{n}$  &     \multirow{2}{*}{$\alpha_1$}   
  	& $\beta_{2}^{n+1}$ 
  \\    \cline{3-3}     \cline{5-5} 
	& & $\beta_{1}^{n-1}$ && $\beta_{n-1}^{n+1}$
\\   \cline{2-5}
	& \multirow{2}{*}{$\beta_{2}^{n+1}$} & $\beta_{2}^{n}$
	&  \multirow{2}{*}{$\alpha_1$}    & $\beta_{1}^{n}$  
\\  \cline{3-3} \cline{5-5}
	&&$\beta_{3}^{n+1}$&&$\beta_{1}^{3}$ 
\\  
\hline  
	\multirow{2}{*}{$n$}&\multirow{2}{*}{$\alpha_1$} 
  	& $\beta_{1}^{n}$  & & $\alpha_{1}^{n}$ 
\\    \cline{3-3}   \cline{5-5}   
	& &$\beta_{2}^{n+1}$ &&$\alpha_{1}^{2}$ 
\\    
 \hline
 \end{tabular}
\end{center}
\end{table}

If $\gamma=\alpha_1$ then $\gamma_{n-1}\in \{ \beta_1^n, \beta_2^{n+1}\}$ 
since there are only two curves of type $(0,n-1)$ in $\kai_{1,n}$.
If $\gamma_{n-1}=\beta_1^n$ then  $\delta_{\gamma}=\alpha_1^n$, and if 
$\gamma_{n-1}=\beta_2^{n+1}$ then $\delta_{\gamma}=\alpha_1^2$.
For each vertex $\gamma=\gamma_k$, the element $\delta_\gamma$, unique in cases $k\neq n-1$, 
is given in the table. 
\end{proof}

\begin{corollary}\label{cor:alpha1->nonadj}
	If $P\subset \kai_{1,n}$ is a linear pants decomposition of 
	$N$ containing $\alpha_1$, then $\phi(P)$ is linear. 
\end{corollary}

\subsubsection {$\phi$ preserves topological types of vertices}

We now prove that the locally injective simplicial map $\phi:\kai_{1,n}\rightarrow\CC(N)$
preserves the topological types of vertices of $\kai_{1,n}$.

\begin{proposition}\label{g1ct}
       For every vertex  $\gamma\in \kai_{1,n}$,
       the topological types of $\phi(\gamma)$ and $\gamma$ are the same.
\end{proposition}
\begin{proof}  
  Suppose that $P=\{ \ga_2,\ga_3,\ldots ,\ga_{n-1},\ga_n =\alpha_1\}$ is a linear pants decomposition contained in 
  $\kai_{1,n}$  such that $\ga_k$ is of type 
  $(0,k)$ and  is connected to $\ga_{k+1}$ in the adjacency graph $\gap$ for $k=2,3,\ldots,n-1$. 
 By Corollary~\ref{cor:alpha1->nonadj},  $\phi(P)$ is a linear pants decomposition of $N$. 
 Since $\ga_2$ and $\ga_n$ have valency one in $\gap$, 
 $\phi(\ga_2)$ and $\phi(\ga_n)$ have valency one  in $\gapf$. 
 By Remark~\ref{valency1}, one of $\phi(\ga_2)$ and $\phi(\ga_n)$ 
 is one-sided and the other is of type $(0,2)$.
	
 Let us consider the linear pants decompositions 
 $P_1=\{ \beta_{1}^{3}, \beta_{1}^{4}, \beta_{1}^{5}, \ldots, \beta_{1}^{n}, \alpha_{1}\}$ and 
 $P_2=\{ \beta_{n-1}^{n+1}, \beta_{n-2}^{n+1},  \ldots,  \beta_{3}^{n+1},    \beta_{2}^{n+1}, 
       \alpha_{1}\}$.
 It follows from the argument in the previous paragraph that the topological types of
 $\phi(\beta_{1}^{3})$ and  $\phi(\beta_{n-1}^{n+1})$ are 
 the same; they are either both one-sided or both of 
 type $(0,2)$. Since they are disjoint and since any two one-sided curves on $N$
 must intersect, these two curves must be of type $(0,2)$. Hence,  $\phi(\alpha_1)$ is one-sided.
  
 Let us now consider a curve $\beta_i^j$ of type $(0,k)$ so that $k=j-i$. 
 By Lemma~\ref{cccc},  
 there is a linear pants decomposition in $\kai_{1,n}$ containing $\beta_i^j$ and $\alpha_1$.
 By Corollary~\ref{cor:alpha1->nonadj},  
 $\phi(P)$ is a linear pants decomposition and by Lemma~\ref{lingraph},
 $\phi(\beta_i^j)$ is of type $(0,k)$.
 
 Finally, for every $\alpha_1^k$, there is a pants decomposition $P$ containing it. 
 Since $\phi(\gamma)$ is two-sided
 for every curve $\gamma \in P$ different than $\phi(\alpha_1^k)$, the curve  
 $\phi(\alpha_1^k)$ must be one-sided.
 
 This finishes the proof of the proposition.
 \end{proof}

\subsection{The case $g\geq2$.} \ 

 Let $g\geq 2$, $g+n\geq 5$ and let $N$ denote the non-orientable surface $N_{g,n}$ 
 of genus $g$ with $n$ holes given in Section~\ref{sec:model}.
 Recall that by Corollary~\ref{toppants}, a top dimensional pants decomposition of $N$
 contains exactly $2g+n-3$ elements.
 Let $\phi:\kaign\rightarrow\CC(N)$ be a locally injective simplicial map.
 Our aim in this subsection is to show that $\phi$ preserves topological types of
 each $\alpha_i$ and of the curves of type $(2,0)$.

\subsubsection{Closed Surfaces} \

Assume that $n=0$ so that $g\geq 5$.  

\begin{lemma}\label{nonadjforclo}
	Let $P$ be a top dimensional pants decomposition of $N$, 
	$\gap$ be its adjacency graph 
	and $\gamma$ be a vertex in $\gap$. The vertex
	$\ga$ is essential one-sided if and only if its valency in $\gap$ is two. 
\end{lemma}
\begin{proof} Since $N$ is closed, by Lemma~\ref{curvetype}, 
	the curve $\ga$ is either essential one-sided or  
	of type $(k,0)$ for some $k$,  where $2 \leq k \leq g-2$.
	
	Suppose that $\ga$ is essential and one-sided. 
	The surface obtained by cutting $N$ along the simple closed curves 
	in $P \setminus \{\ga\}$ is the disjoint union of 
	$g-3$ pairs of pants and a subsurface $X$
	homeomorphic to the real projective plane with two holes. The curve   
	$\ga$ lies on $X$. The holes of $X\subset N$ come from two distinct simple 
	closed curves in $P \setminus \{\ga\}$.  
	Hence,  the valency of $\ga$ in $\gap$ is two.
	
	Let $\delta$ be a simple closed curve in $P$ of type $(k,0)$ for some $k$
	with $2 \leq k \leq g-2$.    
	The  surface $N_\delta$ obtained by cutting $N$ along $\delta$
	has two connected components $X_1$ and $X_2$, 
	one of them is homeomorphic to $N_{k,1}$, the other is homeomorphic to $N_{g-k,1}$. 
	Since $N$ is closed, each of $X_1$ and $X_2$ contains two of elements in $P$ adjacent to  
	$\delta$ in $\gap$. Hence, valency of 
	$\delta$ is four in $\gap$. 
	
	Since a one-sided curve contained in a top dimensional pants decomposition must be essential, 
	the proof of the lemma is complete.
\end{proof}

\begin{lemma}\label{cloct}
	For each $i=1,2,\ldots,g$, 
	the curve $\phi(\alpha_i)$ is essential one-sided.
\end{lemma}
\begin{proof}
       Let $P$ be any pants decomposition in $\kai_{g,0}$ containing all $\alpha_i$
       such that $P$ satisfies the nonadjacency condition.  Note that such a pant decomposition $P$ exists
       and is of top dimensional.
       In $\gap$, $\alpha_i$ has valency two so that the number of vertices in $\gap$ 
       nonadjacent to $\alpha_i$ is $|P|-3$.
       Since $P$ satisfies the nonadjacency condition,
       there are at least $|P|-3$ vertices in $\gapf$ nonadjacent to $\phi(\alpha_i)$.
       As there is no vertex of valency one in $\gapf$, the valency of $\phi(\alpha_i)$
       is exactly two. By Lemma~\ref{nonadjforclo}, $\phi(\alpha_i)$ is one-sided and essential.  
\end{proof}

\subsubsection{Surfaces with $n\geq 1$ holes}
Suppose now that $n\geq 1$ and that $N$ denotes the surface $N_{g,n}$.

\begin{lemma}\label{nonadjforn2}
	Let  $P$ be a top dimensional pants decomposition of $N$, 
	$\gap$ be its adjacency graph 
	and let $\ga$ be a vertex of valency one in $\gap$. 
	Then, $\ga$ is either essential one-sided or of type $(0,2)$. 
\end{lemma}
\begin{proof}
	Assume that $\ga$ is neither essential one-sided 
	nor of type $(0,2)$.  Since 
	$\ga$ is contained in the top dimensional pants decomposition $P$, 
	by Lemma~\ref{curvetype}, $\ga$ is separating and 
	the surface $N_\ga$ obtained by cutting $N$ along $\ga$ 
	is the disjoint union of two subsurfaces of $N$, none of which is a pair of pants. 
	On each of these components, there is a vertex of $P$ adjacent to $\ga$, 
	so that the valency of $\ga$ in $\gap$ is at least two. 
	
	By this contradiction, $\ga$ is 
	either essential one-sided or of type $(0,2)$.	
\end{proof}

\begin{corollary}\label{cor:linpath}
	Let $P$ be a top dimensional pants decomposition of $N$
	and $\gamma_0\in P$ have valency one in $\gap$. Suppose that a vertex $\gamma_k\in P$ 
	is a $k$th linear successor of $\gamma_0$.  If $\gamma_0$ is one-sided, then
	$\gamma_k$ is of type $(1,k)$, and 
	if $\gamma_0$ is of type $(0,2)$, then
	$\gamma_k$ is of type $(0,k+2)$. 
\end{corollary}
 \begin{proof} 
  Let $\gamma_0,\gamma_1,\ldots,\gamma_k$ be 
 a linear path in $\gap$ from $\gamma_0$ to $\gamma_k$ such that each $\gamma_i$
 is an $i$th linear successor of $\gamma_0$. 
 The proof then follows from the fact that for $i=1,2,\ldots, k$, 
 the curves $\gamma_{i-1}$ and $\gamma_i$ bound a pair of pants.  
\end{proof}

\begin{lemma}\label{linpath}
	Let $P\subset\kaign$ be a top dimensional pants decomposition
	satisfying the nonadjacency condition and let $\gamma_0\in P$ have valency one
	in $\gap$. Suppose that 
	$\ga_k$ is a $k$th linear successor of $\ga_0$ in $\gap$.
	 Then $\phi(\ga_0)$ has valency one in the adjacency graph $\gapf$  of $\phi(P)$ and 
	 $\phi(\ga_k)$ is a $k$th linear successor of $\phi(\ga_0)$ in $\gapf$.
\end{lemma}
\begin{proof}
       Let $\gamma_0,\gamma_1,\ldots,\gamma_k$ be a linear path such that each $\gamma_i$
       is an $i$th linear successor of $\gamma_0$.
	Since  $\ga_1$ is the unique vertex  in $\gap$ adjacent to $\ga_0$ and since 
	$\phi$  preserves nonadjacency in $\gap$,  $\phi(\ga_1)$ is the unique vertex
	in $\gapf$ adjacent to $\phi(\ga_o)$.  Hence, $\phi(\ga_0)$ has valency one in $\gapf$. 
	
	The vertex $\ga_i$ is adjacent to $\ga_{i-1}$ and $\ga_{i+1}$ for 
	$i=1,2,\ldots,k-1$. Since $\phi$ preserves nonadjacency in $\gap$, 
	the valency of $\phi(\ga_i)$ is either one or two in $\gapf$. On the other hand 
	since $\gapf$ is connected, $\phi(\ga_i)$ has valency two in $\gapf$. 
	Hence,  the vertex $\phi(\ga_i)$ must be adjacent to 
	the vertices $\phi(\ga_{i-1})$ and $\phi(\ga_{i+1})$. Thus, 
	$\phi(\ga_0),\phi(\ga_1),\ldots,\phi(\ga_k)$ is a linear path, so that 
	$\phi(\ga_k)$ is a $k$th linear successor of $\phi(\ga_0)$.	
\end{proof}

\begin{lemma}\label{g3ct} 
	The curve $\phi(\al_g)$ is one-sided and essential. 
\end{lemma}
\begin{proof}
 For the sake of notational simplicity, let us denote the curve $\beta_{g-1}^k$ by $\beta^k$.
Consider the top dimensional pants decomposition
  \[
    P= \{\al_1,\al_2,\ldots,\al_{g-1},\al_g,\be^{g+1},\be^{g+2},
   \ldots,\be^{g+n},\be^1,\be^2,\ldots,\be^{g-3}\} 
  \]
  in $\kai_{g,n}$,
 so that the vertex $\be^{g+1}$ is the unique vertex connected to $\al_g$ by an edge  
 in the adjacency graph $\gap$ (cf. Figure~\ref{cokgen}). 
	
	Note that
	$\al_g,\be^{g+1}, \be^{g+2}, \ldots,\be^{g+n}$ is 
	a linear path in $\gap$ with $n+1$ vertices such that for $1\leq k\leq n$ the curve 
	$\beta^{g+k}$  is a $k$th linear successor of $\alpha_g$. 
	It is easy to see that the pants decomposition $P$ satisfies the nonadjacency condition.
    \begin{figure}[hbt!]
	\labellist
	\small\hair 2pt
 		\pinlabel $\alpha_1$ at 135 495
 		\pinlabel $\alpha_{g-2}$ at 65 423
 		\pinlabel $\alpha_{g-1}$ at 67 347
 		\pinlabel $\alpha_{g}$ at 133 284
		\pinlabel $\beta^{1}$ at 110 480
		\pinlabel $\beta^{g-3}$ at 82 453
  		\pinlabel $\beta^{g+n}$ at 190 502
 		\pinlabel $\beta^{g+n-1}$ at 250 483
		\pinlabel $\beta^{g+4}$ at 272 455
		\pinlabel $\beta^{g+3}$ at 290 410
		\pinlabel $\beta^{g+2}$ at 286 355
		\pinlabel $\beta^{g+1}$ at 255 305
		\pinlabel $\alpha_{g}$ at 450 494
		\pinlabel $\beta^{g+1}$ at 455 475
		\pinlabel $\beta^{g+2}$ at 455 455
		\pinlabel $\beta^{g+3}$ at 455 433
		\pinlabel $\beta^{g+4}$ at 455 411
		\pinlabel $\beta^{g+n}$ at 455 390
		\pinlabel $\beta^{1}$ at 449 368
		\pinlabel $\beta^{2}$ at 449 346
		\pinlabel $\beta^{3}$ at 449 324
		\pinlabel $\beta^{g-3}$ at 455 304
	 	\pinlabel $\alpha_{g-1}$ at 455 282
		\pinlabel $\alpha_{1}$ at 396 378
		\pinlabel $\alpha_{2}$ at 396 356
		\pinlabel $\alpha_{3}$ at 396 334
		\pinlabel $\alpha_{g-2}$ at 390 292
	\endlabellist
		\centering
		\includegraphics[width=0.7\textwidth]{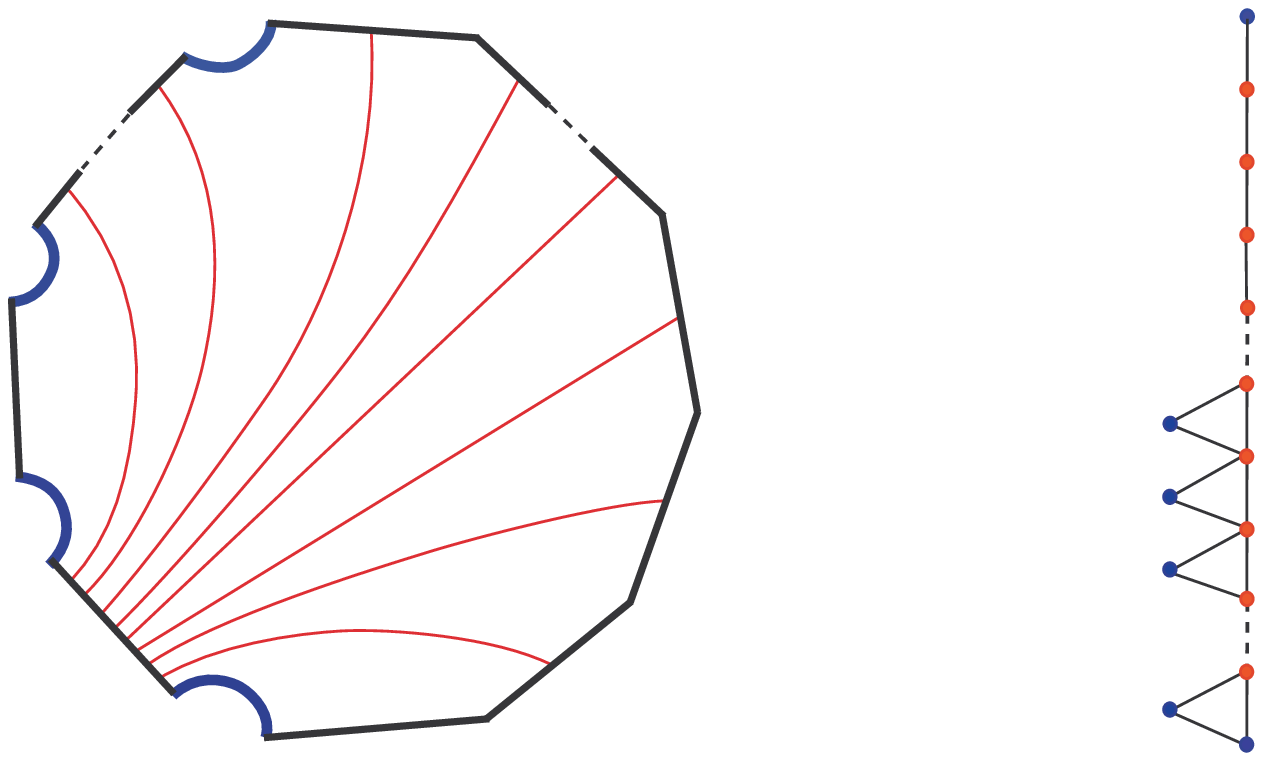}
		\caption{The pants decomposition $P$ containing $\al_1$, and 
			its adjacency graph.}\label{cokgen}
	\end{figure}
	Hence, the map $\phi$ preserves the nonadjacency in $\gap$.
	Since $\al_g$ has valency one in $\gap$,  
	so is $\phi(\al_g)$ in $\gapf$. By Lemma~\ref{nonadjforn2}, $\phi(\al_g)$
	is either essential one-sided or of type $(0,2)$.
	
	Suppose that $\phi(\al_g)$ is of type $(0,2)$. 
	By Lemma~\ref{linpath}, the vertex $\phi(\beta^{g+k})$ 
	is a $k$th linear successor of $\phi(\al_g)$ in the graph $\gapf$. Since 
	$\phi(\al_g)$ is of type $(0,2)$,  the vertex $\phi(\be^{g+n})$ is of type $(0,n+2)$. 
	But there are no simple closed curve of type $(0,n+2)$ on $N$. 
	By this contradiction,  $\phi(\al_g)$ must be essential one-sided.
\end{proof}

We finish this section with the following corollary.

\begin{corollary}\label{all-alpha_i}
	Let $g+n\geq 5$. All of the curves $\phi(\al_1),\phi(\al_2),\ldots,\phi(\al_g)$ are essential one-sided.
\end{corollary} 
\begin{proof}
       The case $g=1$ is proved in Proposition~\ref{g1ct} . So suppose that $g\geq 2$. 
       By Lemma~\ref{g3ct}, $\phi(\alpha_g)$ is essential one-sided. Let $h:N\to N$ be a homeomorphism  with 
       $h(\phi (\alpha_g))=\alpha_g$. Now $\bar{\phi}=h\phi$ restricts to a locally injective simplicial map
       $\link_{\kai_{g,n}} (\alpha_g) \to \link_{\CC(N)} (\alpha_g) $. 
              
       	The surface $N_{\al_g}$ obtained by cutting $N$ along $\al_g$ 
	is a non-orientable surface of 	genus $g-1$ with $n+1$ holes. 
	 Consider the injective simplicial map
	 $\varphi : \link_{\CC(N)} (\alpha_g) \to \CC(N_{\alpha_g})$ induced by cutting
	 $N$ along $\alpha_g$. We then have the following commutative diagram:
	 \begin{equation*}
	\begin{split}
	\xymatrixcolsep{4pc} \xymatrix @-0.1pc
	{
		\link _{\kaign}(\alpha_g)\ar[r]^-{\bar{\phi}}  \ar[d]^-{\varphi} & \link_{\CC(N_{g,n})}(\alpha_g) \ar[d]^-{\varphi}\\
		\kai_{g-1,n+1}\ar[r]^-{\varphi\bar{\phi}\varphi^{-1}} & \CC(N_{\alpha_g}).&
	}
	\end{split}
	\end{equation*}	 
	 The map $\varphi\bar{\phi}\varphi^{-1}$ is locally injective simplicial. By Lemma~\ref{g3ct}, 
	 $\varphi\bar{\phi}\varphi^{-1}$ maps the essential one-sided curve  $\varphi(\alpha_{g-1})$
	 to an essential one-sided curve in $N_{\al}$.
	 It follows that $\bar{\phi}(\alpha_{g-1})$, and hence $\phi(\alpha_{g-1})$, is one-sided and essential.
	  
	  By continuing in this way, one concludes that $\phi(\alpha_k)$ is one sided essential for all
	  $1\leq k \leq g$.
	 \end{proof}

\subsubsection{Curves of type $(2,0)$}
 In the proof of the main theorem, we also need that the map $\phi$ preserves the topological types 
 of the curves of type $(2,0)$ in $\kai_{g,n}$. In the next lemma, we take, as usual, the subscripts and superscripts modulo
$g+n$.

 \begin{lemma} \label{lemma(2,0)}
      Let $g\geq 2$, $n\geq 1$ and $1\leq i\leq g-1$.  If $\beta=\beta_{i-1}^{i+1}\in \kai_{g,n}$
      so that $\beta$ is of type $(2,0)$, then $\phi(\beta)$ is of type $(2,0)$.
 \end{lemma} 
 \begin{proof}
	Consider the top dimensional pants  decomposition
	\[
	  P=\{\al_1,\al_2,\ldots,\al_g,
	  \be_{i-1}^{i+1},\be_{i-1}^{i+2},\ldots,\be_{i-1}^{g+n},\be_{i-1}^1,\be_{i-1}^2,\ldots,\be_{i-1}^{i-3} \}.
	\]
	It is easy to see that the pants decomposition $P$ satisfies the nonadjacency condition. By Lemma~\ref{nonadjacency},
	$\phi$ preserves nonadjacency in the adjacency graph of $P$, i.e. images of nonadjacent vertices in $P$
	are nonadjacent in $\phi(P)$. 
	
	\begin{figure}[h!]
\labellist
\small\hair 2pt
 	\pinlabel $\phi (\beta_{i-1}^{i+1})$ at 200 590
 	\pinlabel $\phi (\beta_{i-1}^{i+1})$ at 420 590
 	\pinlabel $\phi (\beta_{i-1}^{i+1})$ at 200 395
 	\pinlabel $\phi (\beta_{i-1}^{i+1})$ at 423 395
 	\pinlabel $\phi (\alpha_{i})$ at 80 473
 	\pinlabel $\phi (\alpha_{i})$ at 394 473
 	\pinlabel $\phi (\alpha_{i})$ at 98 278
 	\pinlabel $\phi (\alpha_{i})$ at 394 278
  	\pinlabel $\phi (\alpha_{i+1})$ at 230 473
 	\pinlabel $\phi (\alpha_{i+1})$ at 533 473
 	\pinlabel $\phi (\alpha_{i+1})$ at 230 278
 	\pinlabel $\phi (\alpha_{i+1})$ at 533 278
 \endlabellist		
  \centering
		\includegraphics[width=0.6\textwidth]{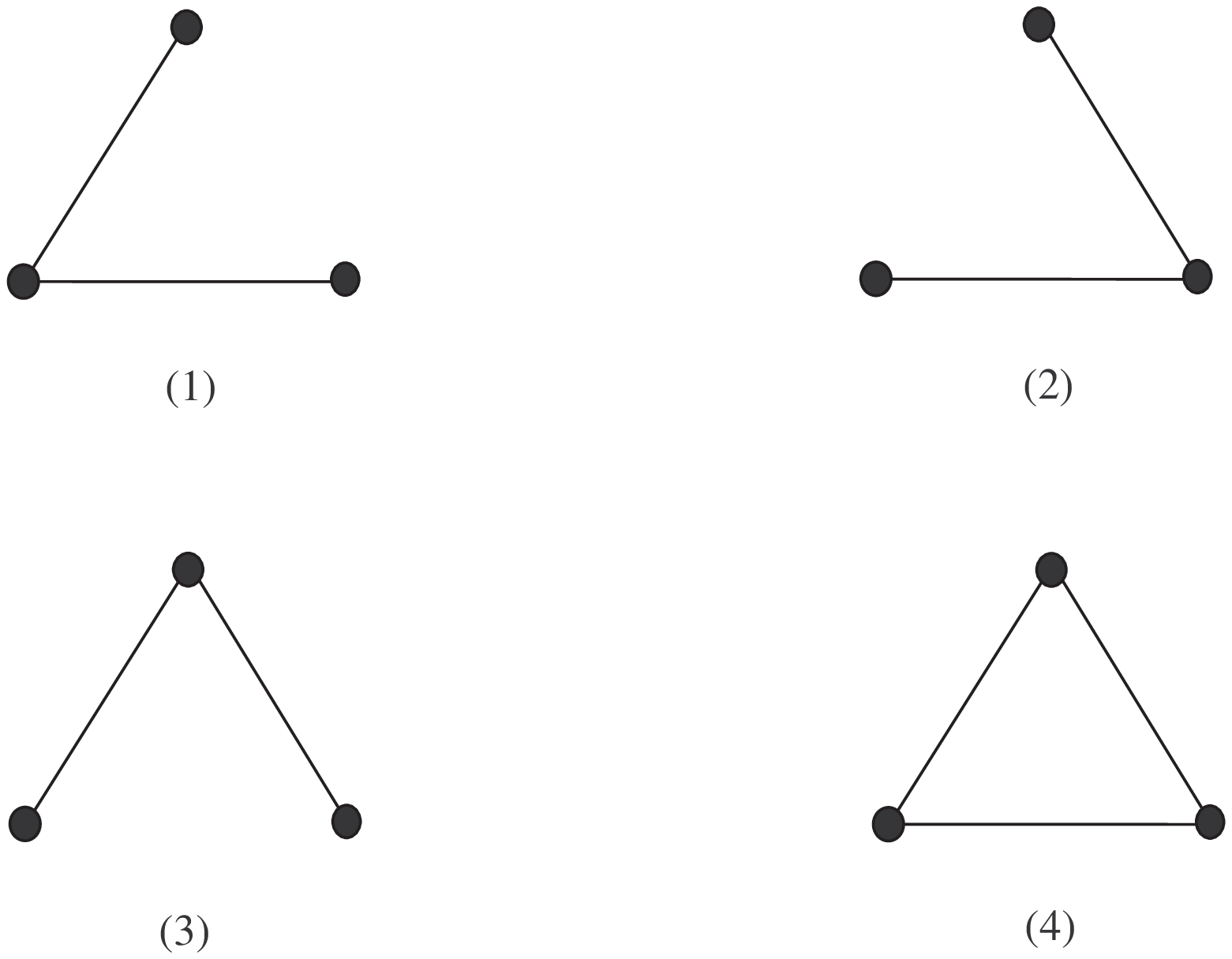}\\
		\caption{Possible configurations for $\phi(\al_i)$, $\phi(\al_{i+1})$ and 
			$\phi(\be_{i-1}^{i+1})$ in $\gapf$}\label{three}
	\end{figure}
	
	We note  that the vertices $\al_i,\al_{i+1}$ and $\be=\be_{i-1}^{i+1}$ 
	form a triangle in $\gap$ and $\be$ is the unique two-sided simple closed curve that is adjacent 
	to both $\al_i$ and $\al_{i+1}$ in $\gap$. 
	Since $\phi(\al_j)$ is essential one-sided for each $j=1,2,\ldots,g$ 
	and $\be$ is disjoint from all $\al_j$, $\phi(\be)$ is disjoint from 
	all $\phi(\al_j)$, and hence it is a two-sided curve. 
	Since $\phi$ preserves nonadjacency in $\gap$,  $\phi(\al_i)$ and $\phi(\al_{i+1})$ might only be 
	adjacent to each other and $\phi(\beta)$.
 	If it can be shown that the vertices $\phi(\al_i)$, $\phi(\al_{i+1})$ and 
	$\phi(\be)$ form a triangle in $\gapf$, 
	one can conclude easily that $\phi(\be)$ is of type $(2,0)$. 
	
	The pants decomposition $\phi(P)$ is top dimensional, so that its adjacency graph $\gapf$ 
	is connected.  The possible full subgraph of $\gapf$ with vertices   the vertices $\phi(\al_i)$, $\phi(\al_{i+1})$ and 
	$\phi(\be)$ is one of the four graphs  given in Figure~\ref{three}. 

	It is clear that in a pants decomposition, if two one-sided vertices $\gamma_1,\gamma_2$
	are adjacent, then any vertex adjacent to $\gamma_1$ is 
	also adjacent to $\gamma_2$.
	Since $\phi(\al_i)$ and $\phi(\al_{i+1})$ are 
	essential one-sided simple closed curves, the cases $(1)$ and $(2)$
	in Figure~\ref{three} is impossible. 
	
	We now rule out the case $(3)$.  If they form such a graph, then the surface $N$ is a Klein bottle 
	with two holes since $\phi(\alpha_i)$ and $\phi(\alpha_{i+1})$ is not adjacent to any vertex 
	different from $\phi(\beta)$.  But we have the assumption $g+n\geq 5$.
	
	Thus we have proved that the only possible configuration is $(4)$ and hence $\phi(\beta)$ 
	bounds a Klein bottle with one hole, i.e., it is of type $(2,0)$, finishing the proof of the lemma.
 \end{proof}

\section{Proof of the Main Result} \label{sec=mainthm}
In this section, we prove the main result of this paper, Theorem~\ref{mainxx}, also stated as Main Theorem in the introduction.
Let $N$ denote the non-orientable surface $N_{g,n}$.

\begin{lemma}
       \label{idenonstar}
 	Let $g+n\geq 5$.  If a locally injective simplicial map 
	$\phi : \kai_{g,n} \to \CC(N)$
	fixes every vertex of the star of $\alpha_i$ in $\kai_{g,n}$ for some $1\leq i\leq g$, then
	$\phi$ fixes every vertex of $ \kai_{g,n}$.
\end{lemma}
\begin{proof}
 All vertices of $ \kai_{g,n}$ but $\alpha_i^j$ are in the star of  $\alpha_i$. So we need to prove that
 $\phi(\alpha_i^j)=\alpha_i^j$ for all $j$. With the agreement that we take the indices modulo $g+n$,
 let us fix some $j$ with $j\notin\{ i-1,i\}$ and
 define the following subsets of $\kai_{g,n}$:
 \begin{enumerate}
     	\item $P_1=\{ \alpha_1, \alpha_2,\ldots, \alpha_{i-1}, \alpha_{i+1}, \alpha_{i+2},\ldots,\alpha_g   \}$,
 	\item $P_2=\{ \beta_{i-1}^{j}, \beta_{i-1}^{j+1}, \ldots, \beta_{i-1}^{g+n}, 
		\beta_{i-1}^{1},\beta_{i-1}^{2},\ldots,\beta_{i-1}^{i-3} \}$,
 	\item $P_3=\{ \beta_{i}^{i+2}, \beta_{i}^{i+3}, \ldots, \beta_{i}^{j} \}$.
 \end{enumerate}
 The surface $N$ cut along the curves in $P_1\cup P_2\cup P_3$ is a disjoint union
 of a number of pair of 
 pants and a non-orientable surface $U$ of genus one with two boundary components.
 There are only two nontrivial curves on $U$; $\alpha_i$ and $\alpha_i^j$. Since $\phi(\alpha_i^j)$ 
 lies on $U$ and is different from $\alpha_i$, we must have $\phi(\alpha_i^j)=\alpha_i^j$.
\end{proof}

\begin{theorem}\label{mainxx}
	Let $g+n\neq 4$.  The finite complex $\kai_{g,n} $ is rigid in
	$\CC(N)$. That is, any locally injective 
	simplicial map $\phi:\kai_{g,n} \rightarrow \CC(N)$ 
	is induced from an element  $F \in \Mod(N)$. 
	The mapping class $F$ is unique up to composition 
	with the involution that interchanges the two faces of the model of $N$.
\end{theorem}
\begin{proof}
       It is easy to see that the sets 
     \begin{itemize}
      \item[]  $\kai_{1,0}=\{\al_1\}=\CC(N_{1,0})$,
       \item[] $\kai_{1,1}=\{\al_1\}=\CC(N_{1,1})$,
       \item[]  $\kai_{1,2}=\{\al_1,\al_1^2\}=\CC(N_{1,2})$,
       \item[]  $\kai_{2,0}=\{\al_1,\al_2\}$,
       \item[]  $\kai_{2,1}=\{\al_1,\al_2\}$,
       \item[]  $\kai_{3,0}=\{\al_1,\al_2,\al_3\}$
        \end{itemize}   
   are rigid. Hence suppose that $g+n\geq 5$.

	The proof is by induction on the genus $g$. Recall that we consider a sphere 
	as a non-orientable surface of genus $0$.  The base step, the case $g=0$ and $n\geq 5$,  
	is proved in~\cite[Theorem~3.1]{aramlei13}, stated as Theorem~\ref{thm:arle}
	above.
	So assume that $g\geq1$ and that any locally injective simplicial map 
	$\kai_{g-1,m} \rightarrow \CC(N_{g-1,m})$ 
	is induced from a homeomorphism $N_{g-1,m} \rightarrow N_{g-1,m}$ if $(g-1)+m\geq 5$.
	
	Let $\phi:\kai_{g,n} \rightarrow \CC(N)$ be a locally injective simplicial 
	map. The curve $\phi(\al_g)$ is a 	one-sided essential simple closed curve 
	by Lemma~\ref{cloct} if $n=0$ and by Corollary~\ref{all-alpha_i} if $n\geq 1$.
	By the classification of surfaces, 
	there exists a homeomorphism 
	$h:N \rightarrow N$ such that $h(\phi(\al_g))=\al_g$.  Since $\phi$ is induced 
	from a homeomorphism of $N$ if and only if $h\phi$ is, we 
       may assume, by replacing $h\phi$ by $\phi$, that $\phi(\al_g)=\al_g$.
       The map $\phi$ then induces a locally injective simplicial map 
       \[
       \link_{\kai_{g,n}}(\alpha_g)\to \link_{\CC(N)} (\alpha_g).
       \]

	The surface $N_{\al_g}$ obtained by cutting $N$ along $\al_g$ 
	is a non-orientable surface of 	genus $g-1$ with $n+1$ holes. As in the proof of
	Corollary ~\ref{all-alpha_i}, consider the injective simplicial map
	 $\varphi : \link_{\CC(N)} (\alpha_g) \to \CC(N_{\alpha_g})$ induced by cutting
	 $N$ along $\alpha_g$ and the commutative diagram
	 \begin{equation*}
	\begin{split}
	\xymatrixcolsep{4pc} \xymatrix @-0.1pc
	{
		\link _{\kaign}(\alpha_g)\ar[r]^-{\phi}  \ar[d]^-{\varphi} & \link_{\CC(N)}(\alpha_g) \ar[d]^-{\varphi}\\
		\kai_{g-1,n+1}\ar[r]^-{\bar{\phi}} & \CC(N_{\alpha_g}) ,&
	}
	\end{split}
	\end{equation*}	 
	 so that $\bar{\phi}=\varphi\phi\varphi^{-1}$ is a locally injective simplicial map. 
	It is easy to see that the map $\varphi$ on the left-hand side  of the diagram is an isomorphism.
	 
	 By the induction hypothesis, since $(g-1)+(n+1)\geq 5$,
	 there is a homeomorphism $f: N_{\alpha_g}\to  N_{\alpha_g}$ such that 
	 $f(\gamma')=\bar{\phi}(\gamma')$ for every curve $\gamma'\in \kai_{g-1,n+1}$ 
	 on $ N_{\alpha_g}$. That is, $f(\varphi(\gamma))= \varphi(\phi(\gamma))$ for every 
	 vertex $\gamma \in \link _{\kaign}(\alpha_g)$.

	 The surface $ N_{\alpha_g}$ has a distinguished boundary component
	  $\Delta$, coming from $\al_g$. We claim that the homeomorphism $f$ 
	  maps $\Delta$ to itself.  
	
	Case $g=1$: Suppose that $f(\Delta)\neq \Delta$. 
	 The curve $\varphi(\beta_1^{n})$ bounds a disc with two holes $\Delta$ and $z_n$ in 
	 $N_{\alpha_1}$, and $\varphi(\beta_2^{n})$ bounds a disc with two holes $\Delta$ and 
	 $z_1$ 
	 (Here, $z_i$ represents the boundary of a small neighborhood of the $i$th puncture.). 
	 Thus $f(\varphi(\beta_1^{n}))=\varphi(\phi( \beta_1^{n}))$ bounds a disc with 
	 two holes $f(\Delta)$ and $f(z_n)$ in	$N_{\alpha_1}$ and 
	 $f(\varphi(\beta_2^{n}))=\varphi(\phi( \beta_2^{n}))$ bounds a disc with 
	 two holes $f(\Delta)$ and $f(z_1)$. Note that at least one of $f(z_1)$ and $f(z_n)$ is not 
	 equal to $\Delta$. Without loss of generality assume that $f(z_n)\neq \Delta$, so that
	 none of the two holes on the disk bounded by $\varphi(\phi( \beta_1^{n}))$ is $\Delta$.
	 It follows that $\phi( \beta_1^{n})$ is of type $(0,2)$.
	 
	On the other hand, by Proposition~\ref{g1ct},  the curve $\phi( \beta_1^{n})$ is of 
	type $(0,n-1)$. Since $n\geq 4$, this is a contradiction. Therefore, we have 
	$f(\Delta)=\Delta$. (If $f(z_1)\neq \Delta$, then use the curve $\beta_2^{n+1}$).
		
	Case $g\geq2$: If $n=0$, then trivially  $f(\Delta)=\Delta$.
	Let $n\geq1$. Assume that $f(\Delta)\neq \Delta$, so that $f(\Delta)$ is a boundary 
	component of $N_{\alpha_g}$ coming from a boundary component of $N$.
	Consider the curve $\beta_{g-2}^{g}$ in the link of $\alpha_g$. 
	Since $\varphi(\beta_{g-2}^{g})$ is of type $(1,1)$ in $N_{\alpha_g}$,
	$f(\varphi (\beta_{g-2}^{g}))=\varphi (\phi(\beta_{g-2}^{g}))$ is of type $(1,1)$ 
	in $N_{\alpha_g}$. Since the hole  on the genus $1$ component of 
	$N_{\alpha_g}-\{ \varphi (\phi(\beta_{g-2}^{g})) \}$ comes from a hole of $N$, 
	$\phi(\beta_{g-2}^{g})$ is of type $(1,1)$ in $N$.  

	On the other hand, since $\beta_{g-2}^{g}$ is of type
	$(2,0)$, so is $\phi(\beta_{g-2}^{g})$ by Lemma~\ref{lemma(2,0)}.  Since $N$ is not a
	non-orientable surface of genus $3$ with one hole, this is a contradiction. 
	By this contradiction we must have $f(\Delta)=\Delta$.
	
	By applying a suitable isotopy on the regular neighborhood of the
	boundary component $\Delta$, 
	one may assume that $f$ sends antipodal points on $\Delta$ to antipodal points. 
	Hence, $f$ descends a homeomorphism $F:N \rightarrow N$ with the property that
	$F(\alpha_g)=\alpha_g$ and $F(\gamma)=\phi(\gamma)$ 
	for every vertex $\gamma \in \link_{\kai_{g,n}}(\alpha_g) $ 

      Now, the automorphism $F^{-1}\phi$ of $\CC(N_{g,n})$ fixes
      $\alpha_g$ and all curves in $\kai_{g,n} $ disjoint from $\alpha_g$. By 
       Lemma~\ref{idenonstar}, $F^{-1}\phi$ fixes every vertex of $\kai_{g,n} $.
       Thus, $\phi$ is induced by the homeomorphism $F:N\to N $.
       
       If $G$ is another homeomorphism inducing $\phi$, then $G^{-1}F$ fixes all curves of
       $\kai_{g,n} $. It can be seen easily that the stabilizer of $ \kai_{g,n} $ 
       in the mapping class group is a cyclic subgroup of order two generated by the
       involution interchanging the two copies of $R$ used as a model on $N$.
       
       The proof of the theorem is now complete.
\end{proof}

\begin{corollary} 
	Let $g+n\neq 4$. Every locally injective simplicial map $\kai_{g,n} \rightarrow \CC(N)$
	is induced by a simplicial automorphism $\CC(N) \rightarrow \CC(N)$.
\end{corollary}
\begin{proof}
      By Theorem~\ref{mainxx},
      a locally injective simplicial map $\phi:\kai_{g,n} \rightarrow \CC(N)$ is induced by a homeomorphism of $N$. 
      This homeomorphism induces a simplicial automorphism of $\CC(N) \rightarrow \CC(N)$ that coincides with 
      $\phi$ on $\kai_{g,n}$. 
\end{proof}

\section{$\kaign$ for $g+n=4$}\label{sec:son}
  	In this section,  we show that the set $\kai_{g,n}$ is not rigid in the case 
	$g+n=4$. Let us define the following subsets $A_{g}$ and $B_{g}$ of $\kai_{g,n}$
	such that $A_{g} \cup B_{g}=\kai_{g,n}$:
     \[
\begin{array}{ll}
  A_{1}=\{\al_1, \al_{1}^2,\be_2^4\},   & B_{1}=\{ \al_{1}^3,\be_1^3\},  \\
  A_{2}=\{\al_1,\al_2, \al_{1}^2,\al_{2}^4,\be_2^4\}, &    B_{2}=\{ \al_1^3,\al_2^3,\be_1^3\},  \\
  A_3=\{\al_1,\al_2,\al_3, \al_1^2,\al_2^4,\al_3^4,\be_2^4\},  &   B_{3}=\{ \al_1^3,\al_2^3,\al_3^1,\be_1^3\},\\
  A_{4}=\{\al_1,\al_2,\al_3,\al_4, \al_1^2,\al_2^4,\al_3^4,\al_4^2,\be_2^4\}, &
      			B_{4}=\{ \al_1^3,\al_2^3,\al_3^1,\al_4^1,\be_1^3\}.   
\end{array}
\]
  
   If $f:N_{g,n}\to N_{g,n}$ is a homeomorphism fixing the one-sided curves $\al_1,\al_2,\ldots,\al_g$,
   then the map $\phi: \kai_{g,n}\to\CC(N_{g,n})$ 
   defined by 
   \[ \phi(\gamma)= 
   \begin{cases} 
      \gamma & \mbox{if } \gamma\in A_{g} \\
        f(\gamma) & \mbox{if } \gamma\in B_{g} \\
   \end{cases}
\]   
   is a locally injective simplicial map.   
   For almost any choice of $f$ (for instance, $f$ is the product 
   $t_{\beta_2^4}t_{\beta_1^3}t_{\beta_2^4}$ of the Dehn twists), the map $\phi$ is 
   not induced by a homeomorphism of the surface.

 Of course, this leaves the question of the existence of a finite rigid set open for the cases 
	$g+n=4$.

\vspace{0.2in}

\end{document}